\documentclass[a4paper,11pt]{article}
\usepackage{etex}
\usepackage{etex}
\reserveinserts{28}
\usepackage{times}
\usepackage{vmargin}
\usepackage[utf8x]{inputenc}
\usepackage{amssymb}
\usepackage{amsthm}
\usepackage{amsmath}
\usepackage{amsfonts}
\usepackage{setspace}  
\usepackage{anysize}
\usepackage{subfigure}
\usepackage[all,2cell]{xy} \UseAllTwocells \SilentMatrices
\usepackage{caption}
\usepackage[square,sort,comma,numbers]{natbib}
\usepackage{epsfig}
\usepackage{multicol}
\usepackage{blindtext}
\usepackage{verbatim}
\usepackage{floatrow}
\usepackage{pgf,tikz}
\usetikzlibrary{arrows, automata}
\newfloatcommand{capbtabbox}{table}[][\FBwidth]
\setpapersize{A4} 
\setmarginsrb{1cm}{1cm}{1cm}{.7cm}{1pt}{27pt}{3pt}{27pt}
\theoremstyle{plain}
\newtheorem{theorem}{Theorem}[section]
\newtheorem{lemma}[theorem]{Lemma}
\newtheorem{corollary}[theorem]{Corollary}
\newtheorem{proposition}[theorem]{Proposition}
\theoremstyle{definition}
\newtheorem{definition}[theorem]{Definition}
\newtheorem{example}[theorem]{Example}
\theoremstyle{remark}
\newtheorem{remark}[theorem]{Remark}
\theoremstyle{Conclusion}

\newcommand{\Z}{\ensuremath{\mathbb{Z}_2^n}}
\newcommand{\Aut}{\ensuremath{\mathrm{Aut}(\mathcal{F}}}
\newcommand{\rvline}{\hspace*{-\arraycolsep}\vline\hspace*{-\arraycolsep}}

\title{Weakly  Equivariant Classification of Small Covers over a Product of Simplicies}
\author{\hspace*{-10pt}
	\begin{minipage}[t]{2.7in} \normalsize \baselineskip 12.5pt
	\centerline{Aslı GÜÇLÜKAN İLHAN}
	\centerline{Dokuz Eyl\"{u}l University}
	\centerline{Department of Mathematics}
	\centerline{T{\i}naztepe, \.{I}ZM\.{I}R}
	\centerline{TURKEY}
	\centerline{asli.ilhan@deu.edu.tr}
	\end{minipage} \kern 0in
	\begin{minipage}[t]{2.7in} \normalsize \baselineskip 12.5pt
	\centerline{S. Kaan G\"{U}RB\"{U}ZER}
		\centerline{Dokuz Eyl\"{u}l University}
		\centerline{Department of Mathematics}
		\centerline{T{\i}naztepe, \.{I}ZM\.{I}R}
		\centerline{TURKEY}
		\centerline{kaan.gurbuzer@deu.edu.tr}
	\end{minipage}
}

\begin{document}
\maketitle

\begin{abstract}
	Given a dimension function $\omega$, we define a notion of an $\omega$-vector weighted digraph and an $\omega$-equivalence between them. Then we establish a bijection between the weakly $(\mathbb{Z}/2)^n$-equivariant homeomorphism classes of small covers over  $\Delta^{n_1}\times\cdots \times \Delta^{n_k}$ and the set of $\omega$-equivalence classes of $\omega$-vector weighted digraphs with $k$-labeled vertices. As an example, we obtain a formula for the number of weakly $(\mathbb{Z}/2)^n$-equivariant homeomorphism classes of small covers over  a product of three simplices.
\end{abstract}



\section{Introduction}

Let $P$ be a simple convex polytope of dimension $n$. A small cover over $P$ is an $n$-dimensional smooth closed manifold $M$ with a locally standard $\mathbb{Z}_2^n$-action whose orbit space is $P$.  Two small covers over $P$ are said to be Davis-Januskiewicz equivalent if there is a weakly $\Z$-equivariant homeomorhism between them covering the identity on $P$. For every small cover $M$ over $P$, there is an associated function from the set of codimension one faces of $P$ to $\Z$ called a characteristic function. Classifications of small covers over $P$ up to Davis-Januskiewicz equivalence, $\Z$-equivariant homeomoprhism or weakly $\Z$-equivariant homeomorphism are closely related with the set of characteristic functions over $P$. This problem has been studied by many authors (\cite{Altunbulak-GuclukanIlhan, Cai-Chen-Lu, Main, Garrison-Scott}).

There are two natural group actions on the set of characteristic functions. The first one is the left free action of general linear group over $\mathbb{Z}_2$ and the second one is the right action of the group of automorphisms of the face poset of $P$. The orbit space of the first one is in one-to-one correspondence with the set of Davis-Januskiewicz equivalence classes of small covers over $P$ (\cite{Davis-Januszkiewicz}). There is a one-to-one correspondence between the orbit space of the second one and the set of $\Z$-equivariant homeomorphism classes of small covers over $P$ (\cite{Lu-Masuda}). As a combination of these results, there is a bijection between the double coset of these actions and the set of weakly $\Z$-equivariant homeomorphism classes of small covers.  These correspondences give interesting connections between topology and combinatorics.

In \cite{Main}, Choi shows that there is a bijection between Davis–Januszkiewicz equivalence classes of small covers over an $n$-cube and the set of acyclic digraphs with $n$-labeled vertices. Using this, one can obtain a bijection between the weakly $\Z$-equivariant homeomorphism classes of small covers over an $n$-cube and the orbit space of the action of $\mathbb{Z}_2 \wr S_n$ on acyclic digraphs with $n$ labeled vertices given by local complementation and reordering of vertices (\cite{GuclukanIlhan}). Local complementation is also appear in the classification of small covers over an $n$-cube up to homeomorphism (\cite{Choi-Masuda-Oum}). To generalize this result to small covers over a product of simplices, we introduce the notion of a $\omega$-vector weighted digraph for a given dimension function $\omega: \{1,2,\cdots,l\} \rightarrow \mathbb{N}$. A $\omega$-vector weighted digraph is a digraph with labeled vertices $\{v_1,\cdots, v_l\}$ where each edge directed from $v_i$ has a non-zero vector in $\mathbb{Z}_2^{\omega(i)}$.  It turns out that there is a bijection between Davis–Januszkiewicz equivalence classes of small covers over a product of simplices, namely $\Delta^{\omega(1)}\times \cdots \times \Delta^{\omega(l)}$, and the set of acyclic $\omega$-vector weighted digraphs. Hence the number of Davis–Januszkiewicz equivalence classes of small covers over a product of simplices is given by the formula (\ref{formulaChoi}) in Proposition \ref{prop:w-vector}. This formula was first obtained by Choi \cite{Main} by defining a surjection to the set of underlying digraphs and counting the sizes of preimages.

We observe that the action of the automorphism group on Davis-Januskiewicz equivalence classes corresponds to three operations on $\omega$-vector weighted digraphs.  The first two are reordering vertices that has the same image under the dimension function and permuting the weights of the edges from a fixed vertex. The third one is a generalization of the local complementation at a vertex $v_i$ which we called $(\sigma,k)$-local complementation since a permutation $\sigma\in S_{\omega(i)}$ and an integer $1\leq k \leq \omega(i)$ are also involved. We say that two $\omega$-vector weighted digraphs are $\omega$-equivalent if one can be obtained from the other by applying a sequence of these operations. Hence there is a bijection between the set of weakly $\Z$-equivariant homeomorphism classes of small covers over a product of simplices and the set of $\omega$-equivalence classes of acyclic $\omega$-vector weighted digraphs. Using this bijection, we give a formula for the number of weakly $\mathbb{Z}_2^n$-equivariant homeomorphism classes of small covers over a product of three simplices. These numbers are closely related with the number of permutations whose cycle decompositions have certain types. For this reason, we give some formulas involving the number of permutations of a certain type in Section \ref{sect:permutations}.

\section{Preliminaries}
\label{sect:Preliminaries}

Let $P$ be a simple convex polytope of dimension $n$ and $\mathcal{F}(P)=\{F_1, \dots,F_m\}$ be the set of facets of $P$. A function $\lambda: \mathcal{F}(P) \rightarrow \mathbb{Z}_2^n$ satisfying the non-singularity condition that
 $$F_{i_1} \cap \cdots \cap F_{i_n}\neq \emptyset \ \ \ \Rightarrow \\ \langle \lambda(F_{i_1}), \dots, \lambda(F_{i_n})\rangle=\Z$$
is called a characteristic function. For any $p \in P$, let $\Z(p)$ be the subgroup of $\Z$ generated by $\lambda(F_{i_1}), \dots, \lambda(F_{i_k})$ where the intersection $ \overset{k}{\underset{j=1}{\cap}} F_{i_j}$ is the minimal face containing $p$ in its relative interior. Then the manifold $M(\lambda)= (P \times \Z) / \sim$ where $$(p,g)\sim(q,h) \ \mathrm{if} \ p=q \ \mathrm{and} \ g^{-1}h \in \Z(p)$$ is a small cover over $P$. 

\begin{theorem}\label{thm:DJ-class}\cite{Davis-Januszkiewicz} For every small cover $M$ over $P$, there is a characteristic function $\lambda$ with $\Z$-homeomorphism $M(\lambda) \rightarrow M$ covering the identity on $P$.
\end{theorem}

Let $\Lambda(P)$ be the set of all characteristic functions on $P$. It is well-known that certain group actions on $\Lambda(P)$ can be used to classify small covers over $P$ up to associated equivalences. For example the group $GL(n, \mathbb{Z}_2)$ acts freely on $\Lambda(P)$ by $g \cdot \lambda=g \circ \lambda$.  For any $\lambda, \lambda'$ in $\Lambda(P)$, the small covers $M(\lambda)$ and $M(\lambda')$ are Davis-Januszkiewicz equivalent if and only if there is a $g\in GL(n, \mathbb{Z}_2)$ such that $g \cdot \lambda=\lambda'$. Therefore the set of Davis-Januszkiewicz equivalence classes of small covers over $P$ corresponds bijectively to the coset $GL(n,\mathbb{Z}_2) \backslash \Lambda(P)$ by the above Theorem.

Another action on $\Lambda(P)$ that gives such a classification is the action of the group of automorphisms of the poset set $(\mathcal{F}(P), \subset)$, which is denoted by $\Aut (P))$. The group $\Aut (P))$ acts  $\Lambda(P)$ on right by $\lambda \cdot h=\lambda \circ h$. As shown in \cite{Lu-Masuda}, for any $\lambda, \lambda'$ in $\Lambda(P)$, there is a $\Z$-equivariant homeomorphism between small covers $M(\lambda)$ and $M(\lambda')$ if and only $\lambda \cdot h=\lambda'$ for some $h \in \Aut (P))$. Hence there is a bijection between the orbit space of this action and the set of $\Z$-equivariant homeomorphism classes of small covers over $P$  \cite{Lu-Masuda}. By combining this with the above Theorem, Lu and Masuda \cite{Lu-Masuda} obtain the following result.

\begin{theorem}\label{thm:weakly-class}\cite{Lu-Masuda} There is a bijection between the set of weakly $\Z$-equivariant homeomorphism classes of small covers over $P$ and the double coset $GL(n,\mathbb{Z}_2) \backslash \Lambda(P) / \Aut (P))$.
\end{theorem}

\section{$\omega$-vector weighted digraphs}
\label{sect:Vector_weighted_digraphs}

In \cite{Choi-Masuda-Suh}, Choi, Masuda and Suh introduce the notion of a vector matrix to associate a quasitoric manifold over a product of simplices. Given a dimension function $\omega:\{1,2,\cdots,n\} \rightarrow \mathbb{N}$, a vector matrix of size $n$ is a matrix $A=[\mathbf{v_{ij}}]$ whose entries in the $i$-th row are vectors in $\mathbb{Z}^{\omega(i)}$. We denote the $k$-th entry of $\mathbf{v_{ij}}$ by $(\mathbf{v_{ij}})_k$. Choi uses the vector matrices over $\mathbb{Z}_2$ to classify small covers over a product of simplices up to Davis-Januszkiewicz equivalences. More precisely, given a function $\omega:\{1,2,\cdots,n\} \rightarrow \mathbb{N}$, let $M_{\omega}(n)$ be the set of all vector matrices $A=[\mathbf{v_{ij}}]$ of size $n$ over $\mathbb{Z}_2$ whose entries in the $i$-th row are vectors in $\mathbb{Z}_2^{\omega(i)}$ that satisfies the following condition: Every principal minors of  $A_{k_1\cdots k_l}$  is $1$ where $A_{k_1\cdots k_n}$ is the $(n \times n)$-matrix whose $(i,j)$-th entry is $(\mathbf{v_{ij}})_{k_i}$, for any $1\leq k_i \leq n_i$ and $1\leq i \leq n$. In \cite{Main}, Choi shows that there is a bijection between the set $M_{\omega}(n)$ and the Davis-Januszkiewicz equivalence classes of small covers over a product of simplices $P=\Delta^{\omega(1)}\times \cdots\times \Delta^{\omega(n)}$. 

There is a one-to-one correspondence between acyclic digraphs with $n$ labeled vertices and the set of $\mathbb{Z}_2$-matrices of size $n$ all of whose principal minors are $1$ (Theorem 2.2, \cite{Main}). Here the correspondence is obtained by sending an acyclic digraph $G$ to the matrix $A(G)+I_n$ where $A(G)$ is the adjacency matrix of $G$.  The set of such matrices is denoted by $M(n)$. Here we introduce the notion of a $\omega$-vector weighted digraphs to generalize this bijection to $M_{\omega}(n)$. Let $\omega:\{1,2,\cdots,n\} \rightarrow \mathbb{N}$ be a function.

\begin{definition} A digraph with labeled vertices $\{v_1,\cdots, v_n\}$ is said to be $\omega$-vector weighted if every edge $(v_i,v_j)$ (if exists) is assigned with a non-zero vector $\mathbf{w_{ij}}$ in $\mathbb{Z}_2^{\omega(i)}$. 
\end{definition}

For convenience, we say that the weight of $(v_i,v_j)$ is the zero vector if and only if there is no edge from $v_i$ to $v_j$.  For the abuse of notation, we also denote the dimension of the weight vector of any edge directed from a vertex $v$ by $\omega(v)$ and we denote the weight of the edge from $u$ to $v$ by $\omega(u,v)$. Note that a $\omega$-vector weighted digraph is indeed a vector weighted digraph when the weight vectors are equidimensional. A $\omega$-vector weighted digraph is called acyclic if it does not contain any directed cycle. 

Let $G=(V,E)$ be a $\omega$-vector weighted digraph with labeled vertices $\{v_1,\cdots, v_n\}$.  We define the adjacency matrix of $G$ as a $(n \times n)$ $\omega$-vector matrix whose $(i,j)$-the entry is the zero vector $\mathbf{0} \in \mathbb{Z}_2^{\omega(i)}$ if there is no edge from $v_i$ to $v_j$ and $\mathbf{w_{ij}}$,  otherwise. We denote it by $A_{\omega}(G)$. For any $1\leq k_i \leq \omega(i)$ and $1\leq i \leq n$, $\big( A_{\omega}(G)\big)_{k_1\cdots k_l}$ is an adjacency matrix of some subgraph of the underlying acyclic digraph and hence we have the following result. 

\begin{proposition}\label{prop:w-vector} There is a one-to-one correspondence between the acyclic $\omega$-vector  weighted digraphs with labeled vertices $v_1,\cdots, v_n$ and $M_{\omega}(n)$. In particular, \begin{eqnarray} \label{formulaChoi}|M_{\omega}(n)|=\sum_{G \in \mathcal{G}_n} \prod_{v_i \in V(G)} (2^{\omega(i)}-1)^{\mathrm{outdeg}(v_i)}\end{eqnarray} where $\mathrm{outdeg}(v_i)$ is the number of edges directed from $v_i$.
 \end{proposition} 

The above formula was first obtained by Choi \cite{Main}. To obtain this formula, Choi counts the preimages of the function $\phi: M_{\omega}(n) \rightarrow M(n)$ which sends $A=[\mathbf{w_{ij}}]$ to an $(n\times n)$-matrix whose $(i,j)$-th entry is $0$ if $w_{ij}$ is the zero vector and $1$, otherwise. Now we give two simple observations about the matrices whose principal minors are $1$.
\begin{lemma}
	If $V=[v_{ij}]$ is an element of $M(n)$ then $$\displaystyle \sum_{\substack{\sigma\in S_n \\ Fix(\sigma)=\emptyset}}v_{1\sigma(1)}v_{2\sigma(2)}\dots v_{n\sigma(n)}=0.$$
\end{lemma}

\begin{proof} We prove by induction on $n.$ The case $n=2$ is trivial. 
The determinant of the matrix $V=[v_{ij}]$ satisfies
	\begin{eqnarray*}
		det \text{ }V= 1&=&
\sum_{\sigma\in S_n}v_{1\sigma(1)}\dots v_{n\sigma(n)}=
1+ \sum_{I\subsetneqq\{1,\dots,n\}}\sum_{\substack{\sigma\in S_n \\ Fix(\sigma)=I}}v_{1\sigma(1)}\dots v_{n\sigma(n)}
	\end{eqnarray*}
and hence
\begin{equation}\label{eqn:1}
	\sum_{I\subsetneqq\{1,\dots,n\}}\sum_{\substack{\sigma\in S_n \\ Fix(\sigma)=I}}v_{1\sigma(1)}\dots v_{n\sigma(n)}=0
\end{equation}
For  $I\neq \emptyset,$ let $\{1,2,\dots,k\}\setminus I=\{j_1,\dots, j_m\}$ with $1\leq j_1<j_2<\dots<j_m\leq n.$ Since $v_{ii}=1$ for all $i,$ we have 

\begin{eqnarray*}
	\sum_{\substack{\sigma\in S_n \\ Fix(\sigma)=I}}v_{1\sigma(1)}\dots v_{n\sigma(n)}&=& 
\sum_{\substack{\beta\in S_m \\ Fix(\beta)=\emptyset}}v'_{1\beta(1)}\dots v'_{m\beta(m)},
\end{eqnarray*}
where $\displaystyle V'=[v_{j_pj_q}]_{1\leq p,q\leq m}$. Here the correspondence between $\sigma$  and $\beta$ is given as follows: $\beta(a)=b$ if and only if $\sigma(j_a)=j_b.$ Since $V'$ is the reduced submatrix of $V$ obtained by removing the rows and columns corresponding to $I,$  $V'\in M(m).$ Hence
the above sum is equal to zero by the induction hypothesis. Then the equation (\ref{eqn:1}) yields the desired result.

\end{proof}

\begin{proposition}\label{prop:productofvs}
	For $n\geq 3,$ let $V=[v_{ij}]\in M(n).$ Then for every $\ B\subseteq\{1,2,\dots,n\}$ with $0\leq |B|\leq n-2$ and for any $i\in \{1,2,\dots,n\}\setminus B,$ the following holds
	\begin{eqnarray}\label{eqn:2}
		\sum_{\{a_1,\dots,a_{n-|B|-1}\}=\{1,\dots,n\}\setminus(B\cup\{i\})}v_{ia_1}v_{a_1a_2}\dots v_{a_{n-|B|-2}a_{n-|B|-1}}v_{a_{n-|B|-1}i}=0
	\end{eqnarray}
\end{proposition}

\begin{proof}
	We prove by induction on $n.$ The case $n=3$ and $|B|=1$ is easily follows. Indeed, without lost of generality, we can assume $b_1=1$ and $i=2.$ Then the left hand side of (\ref{eqn:2}) is $v_{23}v_{32}$ that is a principal minor and hence is equal to zero. 
Let $B
\subset \{1,\dots, n\}
$ with $|B| \leq n-2$ and $i\in \{1,\dots,n\}\setminus B.$ If $B=\emptyset$, the result follows from the above Lemma. Otherwise, fix $b \in B$. Then the left hand side of (\ref{eqn:2}) is of the form
	\begin{eqnarray*}
		\sum_{\{a_1,\dots,a_{(n-1)-(\lambda-1)-1}\}=\{1,2,\dots,\hat{b},\dots,n\}\setminus(B\cup \{i\}\setminus \{b\})} v_{ia_1}v_{a_1a_2}\dots v_{a_{n-\lambda-2}a_{n-\lambda-1}}v_{a_{n-\lambda-1}i}.
	\end{eqnarray*}
Since the principal minors of the matrix obtained by removing the row $b$ and the column $b$ of $V$ is also $1,$ the above sum is zero by induction hypothesis. 

\end{proof}

\section{Small covers over a product of simplices}
\label{sect:SCoverProductofSimp}

Güçlükan İlhan \cite{GuclukanIlhan} shows that there is a bijection between the set of weakly $\mathbb{Z}_2$-equivariant homeomorphism classes of small covers over an $n$-cube and the set of acyclic digraphs with $n$-labeled vertices up to local complementation and reordering vertices. In this section, we generalize this result to a product of simplices using the Therorem \ref{thm:weakly-class}. The automorphism group of a product of simplices depends not only on the dimension of the simplices but also the number of equidimensional simplices appearing in the product (see \cite{Altunbulak-GuclukanIlhan}). For this reason, we let
$$P=\underset{i=1}{\overset{l}{\prod}} P_i, \ \mathrm{where} \ P_i=\Delta_1^{n_i}\times \cdots \times \Delta_{m_i}^{n_i},$$ with $1\leq n_1 <n_2 <\cdots <n_l$ and $\underset{i=1}{\overset{l}{\sum}} n_i m_i=n$. Here the set of facets of $P_i$ is
$$\{f^i_{j,k}=\Delta^{n_i}_1 \times \cdots \times \Delta^{n_{i}}_{j-1}\times \tilde{f}^i_{j, k}\times \Delta^{n_{i}}_{j+1} \times \cdots \times \Delta^{n_i}_{m_i} | \ 1 \leq k \leq n_i+1, \ 1 \leq j \leq m_i \}$$ where $\{\tilde{f}_{j,0}^i,\dots, \tilde{f}^i_{j, n_i}\}$ is the set of facets of the simplex $\Delta^{n_i}_j$. Therefore the set of facets of $P$ is given by
$$\mathcal{F}(P)=\{F_{j,k}^i| \ 1\leq k \leq n_i+1, \ 1\leq j \leq m_i, \ 1\leq i\leq l\}$$
where $F_{j,k}^i=P_1\times \cdots \times P_{i-1} \times f_{j,k}^i \times P_{i+1}\times \cdots \times P_l$. Note that, there are $(n+m)$-facets, where $m=\overset{l}{\underset{i=1}{\sum}}m_i$. The automorphism group of $P$ is $\overset{l}{\underset{i=1}{\prod}} \Big(S_{n_i+1}\wr S_{m_i}\Big)$ where
 $S_{n_i+1}\wr S_{m_i}$ is the wreath product of $S_{n_{i+1}}$ with $S_{m_i}$, where $\sigma \in S_{n_{i+1}}$ sends $F^i_{j,k}$ to $F^i_{j, \sigma(k)}$ and $\mu \in S_{m_i}$ sends $F^i_{j,k}$ to $F^i_{\mu(j),k}$ (see Lemma 3.2 in  \cite{Altunbulak-GuclukanIlhan}). 
 

Let $\displaystyle M^{i}_{j,k}=\Big(\sum_{\alpha=1}^{i-1}m_\alpha n_\alpha\Big)+(j-1)n_i+k$ and $\displaystyle N^{i}_j=\Big(\sum_{\alpha=1}^{i-1}m_\alpha \Big)+j$ where $1\leq i\leq l,$ $1\leq j\leq m_i$ and $1\leq k\leq n_i$. To every characteristic function $\lambda$ over $P$, one can associate an $(n \times (n+m))$ matrix $\Lambda$ whose $M^{i}_{j,k}$-th column is $\lambda(F^i_{j,k})$ and $(n+N^i_j)$-th column is $\lambda(F^i_{j,n_{i}+1})$ .  By reordering the facets and choosing a basis, we can choose a representative of the orbit of $\Lambda$ of the form $(I_n| \Lambda_{\ast})$ where $\Lambda_{\ast}$ is an $(n \times m)$-matrix. Following Choi \cite{Main}, we call $\Lambda_{\ast}$ the reduced submatrix of $\lambda$. As shown in \cite{Main}, $\Lambda_{\ast}$ can be seen as an element of $M_{\omega}(m)$ where $\omega(t)=n_i$ if $t=N^i_j$ for some $1\leq j \leq n_i$. We call the function $\omega$ defined in this way the dimension function of $P$. Since there is a bijection between the Davis-Januszkiewicz equivalence classes of small covers over $P$ and the reduced submatrices, the following result directly follows from the Proposition \ref{prop:w-vector}.
\begin{corollary} The set of Davis-Januszkiewicz equivalence classes of  small covers over $P$ are in one-to-one correspondence with the set of $\omega$-vector weighted acyclic digraphs with $n$ labeled vertices where $\omega$ is the dimension function of $P$. \end{corollary}

An arbitrary element $g \in \Aut (P)) $ can be written as a product of elements of the form $$(1_1,\cdots,1_{i-1}, (\mathrm{id}^i_1,\cdots,\mathrm{id}^i_{j-1}, \sigma_j^i, \mathrm{id}^{i}_{j+1},\cdots, \mathrm{id}^i_{m_i}; \mathrm{id}^i),1_{i+1},\cdots,1_{l})$$ and $$(1_1,\cdots,1_{i-1},(\mathrm{id}^i_1,\cdots, \mathrm{id}^i_{m_i}; \mu^i),1_{i+1},\cdots,1_{l})$$ where $\sigma^i_j\in S_{n_i+1}$ and $\mu^i \in S_{m_i}$. With the abuse of notation, we also denote these elements by $\sigma^i_j$ and $\mu^i$. Here $1_i$, $\mathrm{id}^i_j$, and $\mathrm{id}^i$  are the identity elements in $S_{n_i+1}\wr S_{m_i}$, $S_{n_i+1}$, and $S_{m_i}$, respectively. Let $G=(V,E)$ be an acyclic $\omega$-vector weighted digraph with labeled vertices $v_1,\cdots, v_m$  with weights $\mathbf{w_{pq}}$. Clearly, the corresponding action of  $\mu^i$ sends $G$ to an acyclic $\omega$-vector weighted digraph obtained by reordering the vertices $\{v_p| N^i_1\leq p\leq N^i_{m_i} \}$ by $(\mu^i)^{-1}$ where the weight of the edge $(v_p,v_q)$ in the resulting graph is $\mathbf{w_{\mu(p)\mu(q)}}$.

 When all the simplices are $1$-dimensional, $\sigma^i_j$ is an element of a cyclic group of order $2$ and when it is non-trivial, the corresponding action on the acyclic digraphs is the local complementation at vertex $v_{N^i_j}$ (see \cite{GuclukanIlhan}). The local complementation of a digraph $G=(V(G),E(G))$ at vertex $v$ is the digraph $G \ast v$ with $V(G \ast v)=V(G)$ and $E(G\ast v)$ is the symmetric difference of sets $E(G)$ and $\{(u,w)| \ (u,w) \in N^-_G(v)\times N^+_G(v)\}$ where $N_G^{+}(v)$ is the set of all out-neighbors of $v$ and $N_G^-(v)$ is the set of all in-neighbors of $v$. This notion can be easily generalized to $\omega$-vector weighted digraphs by letting $E(G\ast v)=E(G) \cup N^-_G(v)\times N^+_G(v)   \backslash \{(u,w)  \in N^-_G(v)\times N^+_G(v)| \omega(u,w)+\omega(u,v)=\mathbf{0} \}$ where the weight of the edge $(u,w)$ is $\omega(u,w)+\omega(u,v)$ if $(u,w)\in N^-_G(v)\times N^+_G(v)$ and is $\omega(u,w) $, otherwise. 
 If we assume that weight of the edges in the (unweighted) digraphs  are $1 \in \mathbb{Z}_2^1$, this definition agrees with the standard definition of the local complementation given above. Note that any $\sigma \in S_{\omega(v)}$ also acts on $\omega$-vector weighted digraphs that contains $v$ as a vertex by permuting the coordinates of the weights of the edges from $v$. Now we introduce two generalizations which allow permutation of coordinates of the associated weights.

\begin{definition} Let $v$ be a vertex of a $\omega$-vector weighted digraph $G$ and $\sigma$ be a permutation in $S_{\omega(v)}$. The $\sigma$-local complementation of $G$ at vertex $v$ is obtained by permuting the weights of edges from $v$ in $G\ast v$ by $\sigma$.  We denote the obtained $\omega$-vector weighted digraph by $G\ast_{\sigma} v$.
\end{definition}

\begin{definition} Let $v$ be a vertex of a $\omega$-vector weighted digraph $G$ and $\sigma \in S_{\omega(v)}$. For any $1\leq k \leq \omega(v)$, the $(\sigma,k)$-local complementation of $G$ at vertex $v$ is the $\omega$-vector weighted digraph $G \underset{(\sigma,k)}{\ast}v$ where $V(G \underset{(\sigma,k)}{\ast}v)=V(G)$ and the edge set of $G \underset{(\sigma,k)}{\ast}v$ is the union of sets $E(G)$ and $\{(u,w) \in  N^-_G(v)\times N^+_G(v) | \ (\omega(v,w))_k =1 \ \text{and} \ \omega(u,w)+\omega(u,v)\neq \mathbf{0} \}$. The weight of the edge $(u,w)$ is given by
\begin{enumerate}
\item[i)] $\omega(u,w)+\omega(u,v)$ if $(u,w) \in  N^-_G(v)\times N^+_G(v)$ and $ (\omega(v,w))_k =1$,
\item[ii)] $\sigma \cdot \omega(v,w)$ if $u=v$ and $(\omega(v,w))_k =0$,
\item[iii)] $\sigma \cdot \omega(v,w)+e_{\sigma^{-1}(k)}$ if $u=v$ and $(\omega(v,w))_k =1$ where $e_i$ is the vector in $\mathbb{Z}_2^{\omega(v)}$ all of whose coordinates are 1 except the $i$-th one,
\item[iv)] $\omega(u,w)$, otherwise.
\end{enumerate} 
\end{definition}

\begin{example}

Let $\omega:\{1,2,3,4\} \rightarrow \mathbb{N}$ be defined by $\omega(1)=2$, $\omega(2)=\omega(3)=\omega(4)=3$. Let $G$ be the  $\omega$-vector weighted digraph given below
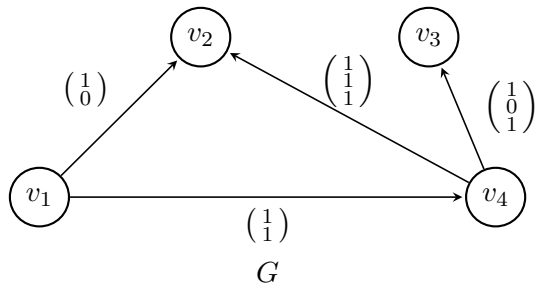
\begin{figure}[h]
	\centering
\begin{tikzpicture}[
	> = stealth, 
	shorten > = 1pt, 
	auto,
	node distance = 3cm, 
	semithick 
	]
	
	\tikzstyle{every state}=[
	draw = black,
	thick,
	fill = white,
	minimum size = 3mm
	]
	
	\node[state] (v1) {$v_1$};
	\node[state] (v2) [above right of=v1] {$v_2$};
	\node (v) [right of=v1]{} ;
	\node[state] (v3) [above right of=v] {$v_3$};
	\node[state] (v4) [right of=v] {$v_4$};
	\node(a) at (3,-1) {$G$} ;
	
	\path[->] (v1) edge node {$\big(\begin{smallmatrix}
		1 \\
		0 
		\end{smallmatrix}\big)$} (v2);

	\path[->] (v4) edge node [near start, swap] {$\Big(\begin{smallmatrix}
		1 \\
		0 \\
		1
		\end{smallmatrix}\Big)$} (v3);
	\path[->] (v1) edge node [swap] {$\big(\begin{smallmatrix}
		1 \\
		1 
		\end{smallmatrix}\big)$} (v4); 
 \path[->] (v4) edge node [above, pos=0.5] {$\Big(\begin{smallmatrix}
 	1 \\
 	1\\
 	1 
 	\end{smallmatrix}\Big)$} (v2);  
	\end{tikzpicture}
\end{figure}

\vspace{0.2 cm}

For $\sigma=(123)\in S_3$, the  $\sigma$-local complementation of $G$ and the $(\sigma,2)$-local complementation of $G$ at vertex $v_4$ are given in the Figure \ref{fig1}.
\begin{figure}[h]
	\centering
	\begin{tikzpicture}[
	> = stealth, 
	shorten > = 1pt, 
	auto,
	node distance = 3cm, 
	semithick 
	]
	
	\tikzstyle{every state}=[
	draw = black,
	thick,
	fill = white,
	minimum size = 3mm
	]
	
	\node[state] (v1) {$v_1$};
	\node[state] (v2) [above right of=v1] {$v_2$};
	\node (v) [right of=v1]{} ;
	\node[state] (v3) [above right of=v] {$v_3$};
	\node[state] (v4) [right of=v] {$v_4$};
		\node(a) at (3,-1) {$G \underset{\sigma}{\ast} v_4$} ;

	\node[state] (v6)[right of=v4] {$v_1$};
	\node[state] (v7) [above right of=v6] {$v_2$};
	\node (w) [right of=v6]{} ;
	\node[state] (v8) [above right of=w] {$v_3$};
	\node[state] (v9) [right of=w] {$v_4$};
	\node(b) at (13,-1) {$G \underset{(\sigma,2)}{\ast} v_4$} ;

	\path[<-] (v3) edge node [below, pos=0.45] {$\Big(\begin{smallmatrix}
		1 \\
		1 
		\end{smallmatrix}\Big)$} (v1);
	
	\path[->] (v4) edge node [near start, swap] {$\Big(\begin{smallmatrix}
		0 \\
		1 \\
		1
		\end{smallmatrix}\Big)$} (v3);
	\path[->] (v1) edge node [swap] {$\big(\begin{smallmatrix}
		1 \\
		1 
		\end{smallmatrix}\big)$} (v4); 
		\path[->] (v1) edge node [] {$\big(\begin{smallmatrix}
		0 \\
		1 
		\end{smallmatrix}\big)$} (v2); 
	\path[->] (v4) edge node [above, pos=0.35] {$\Big(\begin{smallmatrix}
		1 \\
		1\\
		1 
		\end{smallmatrix}\Big)$} (v2);

	\path[->] (v9) edge node[near start, swap] {$\Big(\begin{smallmatrix}
		0 \\
		1 \\
		1
		\end{smallmatrix}\Big)$} (v8);
	\path[->] (v6) edge node {$\big(\begin{smallmatrix}
		1 \\
		1 
		\end{smallmatrix}\big)$} (v9); 
	\path[->] (v9) edge node [above, pos=0.35] {$\Big(\begin{smallmatrix}
		1 \\
		0\\
		0 
		\end{smallmatrix}\Big)$} (v7);
		\path[->] (v6) edge node [] {$\big(\begin{smallmatrix}
		0 \\
		1 
		\end{smallmatrix}\big)$} (v7); 
	\end{tikzpicture}
     \caption{}
    	\label{fig1}
\end{figure}
\end{example}
Let $\overline{\phantom{a}}: S_{n+1} \to S_n$ be defined by $\overline{\sigma}(t)=\sigma(t)$ if $\sigma(t) \neq n+1$, and $\overline{\sigma}(t)=\sigma(n+1)$, otherwise. 

\begin{lemma} If $\sigma^i_j$ fixes $n_{i}+1$ then $\sigma^i_j$ acts by permuting the weights of the edges from $v_{N^i_j}$. If $\sigma^i_j(n_i+1)\neq n_i+1$ then $\sigma^i_j$ act as $(\overline{\sigma}^i_j, \sigma^i_j(n_i+1))$-local complementation at the vertex $v_{N^i_j}$.
\end{lemma}
\begin{proof} For simplicity, we denote $\sigma^i_j$ by $\sigma$. Let $G$ be an acyclic $\omega$-vector weighted digraph whose adjacency matrix is $A_{\omega}(G)=[\mathbf{v_{\alpha\beta}}]$. Let $\lambda$ be the associated characteristic function i.e, $\Lambda_{\ast}=A_{\omega}(G)+[\mathbf{\delta_{ij}}]$ where $(\mathbf{\delta_{ij}})_k$ is $1$ if $i=j$, and $0$ otherwise. Assume that $\sigma \cdot [I_{n\times n}|\Lambda_{\ast}]=[P_\sigma|Q_\sigma]$. Then $\sigma$ sends $A_{\omega}(G)$ to $A_{\omega}(G^{\sigma})$ where $A_{\omega}(G^{\sigma})=P_{\sigma}^{-1}Q_{\sigma}-[\mathbf{\delta_{ij}}]$ since $([P_{\sigma}| Q_{\sigma}])=([I_n|P^{-1}_{\sigma} \cdot Q_{\sigma}])$. Here $P^{-1}_{\sigma}$ is the $\mathbb{Z}_2$-inverse of $P_{\sigma}$. Hence $\sigma$ sends $G$ to the acyclic $\omega$-vector weighted digraph whose adjacency matrix is $A_{\omega}(G^{\sigma})$.
	
	If $\sigma(n_i+1)=\sigma(n_{i}+1)$ then $Q_{\sigma}=\Lambda_{\ast}$ and $P_{\sigma}$ is the block diagonal matrix in which all the blocks are identity matrices of corresponding dimensions except the $N^i_j$-th block that is the permutation matrix of $\overline{\sigma}$. Therefore its inverse is the block diagonal matrix of the same form whose $N^i_j$-th block is the permutation matrix of $\overline{\sigma}^{-1}$. Therefore multiplying it with $Q_{\sigma}$ permutes the $M^i_{j,1},\cdots, M^i_{j,n_i}$-th rows of $\Lambda_{\ast}$ and hence it acts on the coresponding digraph by permuting the weights of the edges from $v_{N^i_j}$.
	
Now suppose that $\sigma(n_i+1)\neq n_i+1$. In this case, we have 
\begin{eqnarray*}
	(P_\sigma)_{pq}=\begin{cases}
		1 \quad& \text{ if } p=q=M^a_{b,c} \ \text{with} \ (a,b) \neq (i,j), \\
		1 \quad& \text{ if }  p=M^{i}_{j,\sigma(k)} \ \text{and}  \ q=M^{i}_{j,k}, \sigma(k)\neq n_{i}+1 \\
		(\mathbf{v_{N^{a}_b,N^{i}_j}})_c & \text{ if } p=M^{a}_{b,c}, q=M^{i}_{j,k}, \text{ and } \sigma(k)=n_i+1  \\
		0 & \text{ otherwise }
	\end{cases}
\end{eqnarray*}
and 
\begin{eqnarray*}
	(Q_\sigma)_{rs}=\begin{cases}
		1 \quad& \text{ if } s=N^{i}_j, r=M^{i}_{j,\sigma(n_i+1)}\\
              0 \quad& \text{ if } s=N^{i}_j, r\neq M^{i}_{j,\sigma(n_i+1)}\\
		(\mathbf{v_{N^a_b,s}})_c & \text{if} \  r=M_{b,c}^a,\ (a,b)\neq(i,j) .	\end{cases}
\end{eqnarray*}
for $1\leq p,q,r\leq n$ and $1\leq s \leq m$. Since $(\mathbf{v_{pp}})_k=1$ for all $p$, the $\mathbb{Z}_2$-inverse of $P_{\sigma}$ is given by
\begin{eqnarray*}
	(P_\sigma^{-1})_{pq}=\begin{cases}
		1 \quad& \text{ if } p=q=M^a_{b,c} \ \text{with} \ (a,b) \neq (i,j), \\
		1 \quad& \text{ if }  p=M^{i}_{j,\sigma^{-1}(k)} \ \text{and}  \ q=M^{i}_{j,k}, k\neq \sigma(n_i+1)\\
		(\mathbf{v_{N^{a}_b,N^{i}_j}})_c & \text{ if } p=M^{a}_{b,c}, q=M^{i}_{j,k}, \text{ and } \sigma(n_i+1)=k  \\
		0 & \text{ otherwise.}
	\end{cases}
\end{eqnarray*} Since   $A_{\omega} \in M_{\omega}(m)$, multiplying with $Q_{\alpha}$ gives us that $P_{\sigma}^{-1}Q_{\sigma}=[\mathbf{v_{\alpha,\beta}'}]$ where
\begin{eqnarray}\label{eq:imageofdigraph}
	(\mathbf{v_{\alpha\beta}'})_{k}=\begin{cases}
		1 \quad& \text{ if } \alpha=\beta, \\
		(\mathbf{v_{N^{i}_j,\beta}})_{\sigma(n_{i}+1)} \quad& \text{ if }  \alpha=N^{i}_{j},  \ k=\sigma^{-1}(n_i+1) \\
		(\mathbf{v_{N^{i}_j,\beta}})_{\sigma(n_{i}+1)}+(\mathbf{v_{N^{i}_j,\beta}})_{\sigma(k)} & \text{ if } \alpha=N^{i}_{j}, \ k\neq \sigma^{-1}(n_i+1)    \\
(\mathbf{v_{\alpha, N^{i}_j}})_{k} \quad& \text{ if }  \alpha\neq N^{i}_{j},  \  \beta=N^{i}_{j}\\
		(\mathbf{v_{\alpha,\beta}})_{k}+(\mathbf{v_{\alpha,N^{i}_j}})_{k}(\mathbf{v_{N^{i}_j,\beta}})_{\sigma(n_{i}+1)} & \text{ otherwise.}
	\end{cases}
\end{eqnarray} Hence $A(G)_{\omega}^{\sigma}$, which is equal to $P_{\sigma}^{-1}Q_{\sigma}-[\mathbf{\delta_{ij}}]$, is the adjacency matrix of the 
the $(\overline{\sigma}^i_j, \sigma^i_j(n_i+1))$-local complementation of $G$ at the vertex $v_{N^i_j}$.

\end{proof}

\begin{remark} Note that if we assume  $(\mathbf{v_{\alpha\beta}})_{n_i+1}$ to be zero, then the formula (\ref{eq:imageofdigraph}) also gives the adjacency matrix of the image of $G$ under the action of $\sigma$ when $\sigma(n_i+1)=n_i+1$. Using this notation, we can also express the adjacency matrix of image of $G$ under the action of product or arbitary $\sigma^i_j$. More precisely, let $\sigma=\underset{i,j}{\prod} \sigma_i^j$. Using the Proposition \ref{prop:productofvs}, one can show that $k$-th coordinate of the $(\alpha,\beta)$-th entry of the adjacency matrix of 
$A_{\omega}(G \cdot \sigma)$ is given by
\begin{eqnarray*}
	\begin{cases}
		0 & \text{ if } \alpha=\beta \\
	\displaystyle 	[V_{\alpha\beta}]^{\sigma}+\sum_{(a_1,\cdots,a_p)\in S_{\alpha,\beta}}[\mathbf{v_{\alpha,a_1}}]^{\sigma}\cdot[\mathbf{v_{a_1,a_2}}]^{\sigma}\cdots[\mathbf{v_{a_{p-1},a_p}}]^{\sigma} \cdot [\mathbf{v_{a_p,\beta}}]^{\sigma} & \text{ if } \alpha=N^{a}_b \text{ and }\sigma^{a}_b(k)=n'_a+1 \\
	\displaystyle 	[\mathbf{v_{\alpha\beta}}]^{\sigma}+[\mathbf{v_{\alpha\beta}}]^{\sigma}_k+\sum_{(a_1,\cdots,a_p)\in S_{\alpha,\beta}}([\mathbf{v_{\alpha,a_1}}]^{\sigma}+[\mathbf{v_{\alpha,a_1}}]^{\sigma}_k)\cdot[\mathbf{v_{a_1,a_2}}]^{\sigma}\cdots[\mathbf{v_{a_{p-1},a_p}}]^{\sigma} \cdot [\mathbf{v_{a_p,\beta}}]^{\sigma} & \text{otherwise.} 
	\end{cases}
\end{eqnarray*} 
where $S_{\alpha,\beta}=\{(a_1,\dots,a_p)| a_k\neq a_l, \ a_k\in \{1,\dots,N^l_m\}\setminus\{\alpha,\beta\}\}$ $(V_{i,j})_{\sigma^{a}_b(k)}$ by $[V_{i,j}]^{\sigma}_k,$ where $i=N^{a}_b$ and $(V_{i,j})_{\sigma^{a}_b(n_a+1)}$ by $[V_{i,j}]^{\sigma},$ where $i=N^{a}_b.$ 
\end{remark}

\begin{definition} We say two $\omega$-vector weighted digraphs are $\omega$-equivalent if one is obtained (up to graph isomorphisms) from the other one by applying a sequence of the following operations:

\begin{enumerate}
\item Reordering vertices whose images under the dimension function are the same,
\item Permutation of the weights of edges from vertex $v$ by an element of $S_{\omega(v)}$,
\item $(\sigma,k)$-local complementation.
\end{enumerate}
\end{definition}
The following theorem directly follows from the above Lemma.

\begin{theorem}\label{thm:main} There is a bijection between the weakly $\Z$-equivariant homeomorphism classes of small covers of $P$ and the set of $\omega$-equivalence classes of $\omega$-vector weighted digraphs with $m$-labeled vertices.
\end{theorem}

It is easy to find the number of $\omega$-equivalence classes of acyclic $\omega$-vector weighted digraphs with $2$-labeled vertices since the operation (3) can be reduced the one that replaces the zeros and ones except the one in the fixed coordinate in the weights of vertices from the fixed vertex. In the following example, we count this number and in the last section we consider the case where digraphs have $3$-vertices. 

\begin{example}\label{ex} Let $P=\Delta^{n_1}\times \Delta^{n_2}$. Then an acyclic $\omega$-vector weighted digraph, where $\omega:\{1,2\} \rightarrow \mathbb{N}$ is the dimension function of $P$ is of the one of the following types

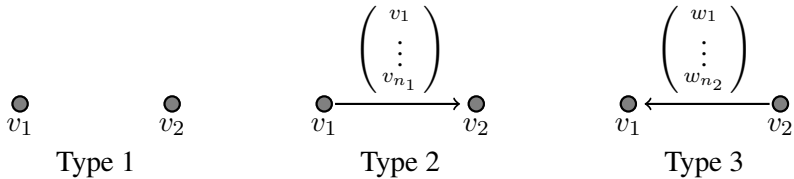
\begin{figure}[h!]
\centering
\begin{tikzpicture}[thick,scale=1,->,shorten >=2pt]
\draw[gray, fill] (0,3) circle [radius=0.1];
\draw (0,3) circle [radius=0.1];
\node () at (0,2.7) {$v_1$};
\draw[gray, fill] (2,3) circle [radius=0.1];
\draw (2,3) circle [radius=0.1];
\node () at (2,2.7) {$v_2$};
\node at (1, 2.2) {Type 1};

\draw[gray, fill] (4,3) circle [radius=0.1];
\draw (4,3) circle [radius=0.1];
\node (a) at  (4,3) {};
\node () at (4,2.7) {$v_1$};
\draw[gray, fill] (6,3) circle [radius=0.1];
\draw (6,3) circle [radius=0.1];
\node (b) at  (6,3) {};
\node () at (6,2.7) {$v_2$};
\path[->] (a) edge node [above] {$\Bigg(\begin{smallmatrix}
		v_1 \\
		\vdots \\
                v_{n_1}
		\end{smallmatrix}\Bigg)$} (b); 

\node at (5, 2.2) {Type 2};

\draw[gray, fill] (8,3) circle [radius=0.1];
\draw (8,3) circle [radius=0.1];
\node (c) at  (8,3) {};
\node () at (8,2.7) {$v_1$};
\draw[gray, fill] (10,3) circle [radius=0.1];
\draw (10,3) circle [radius=0.1];
\node () at (10,2.7) {$v_2$};
\node (d) at  (10,3) {};
\path[->] (d) edge node [above] {$\Bigg(\begin{smallmatrix}
		w_1 \\
		\vdots \\
                w_{n_2}
		\end{smallmatrix}\Bigg)$} (c); 
\node at (9, 2.2) {Type 3};
\end{tikzpicture}
\end{figure}
\noindent where $(v_1,\cdots,v_{n_1}) \in \mathbb{Z}_2^{n_1} \setminus \{\mathbf{0}\}$ and $(w_1,\cdots,w_{n_2}) \in \mathbb{Z}_2^{n_2} \setminus \{\mathbf{0}\}$. Note that if two acyclic $\omega$-vector weighted digraphs on two vertices are isomorphic then they have the same type. Let $n_1\neq n_2$. Then two $\omega$-weighted digraphs $G_1$ and $G_2$ of Type 2 are $\omega$-equivalent if and only if $u_1=u_2$ or $u_1+u_2=n_1-1$ where $u_i$ is the number of zero coordinates of the weight vector of the edge $(v_1,v_2)$ in $G_i$, for $i=1,2$. Therefore the number of equivalence classes of Type 2 is $\lfloor \frac{ n_1+1}{2}\rfloor$. This is also true for the acyclic $\omega$-vector weighted digraphs of Type 3. So there are $1+\lfloor \frac{ n_1+1}{2}\rfloor+\lfloor \frac{ n_2+1}{2} \rfloor$ $\omega$-equivalence classes of acyclic $\omega$-vector weighted digraphs with labeled vertices $v_1,v_2$. Hence there are $1+\lfloor \frac{ n_1+1}{2}\rfloor+\lfloor \frac{ n_2+1}{2} \rfloor$ weakly $\Z$-equivariant homeomorphism classes of small covers over $P=\Delta^{n_1}\times \Delta^{n_2}$ when $n_1\neq n_2$.  

When $n_1=n_2=n$, the reordering of the vertices $v_1$ and $v_2$ is also allowed. Therefore the number of weakly $\Z$-equivariant homeomorphism classes of small covers over $P$ is $1+\lfloor \frac{ n+1}{2}\rfloor$ in this case.
\end{example}

\section{Some results on number of permutations of certain types}
\label{sect:permutations}
In this section, we give some results on the number of permutations of certain types that we need in the next section to find a formula for the number of acyclic  $\omega$-vector weighted digraphs on labeled $3$ vertices up to $\omega$-equivalence. It is well-known that the number of permutations of $n$ elements with $m$ cycle is given by the unsigned Stirling number of the first kind denoted by $c(n,m)$. The Stirling number of the first kind is originally defined as  the coefficient of the expansion of the rising factorial $x^{\bar{n}}$ into powers of $x$, that is,
\begin{eqnarray}\label{rising_Stirling}x^{\bar{n}}=\sum_{m=0}^{n}c(n,m)x^m.\end{eqnarray} They also satisfy the following recurrence relations
\begin{eqnarray*}
	c(n,m)&=& \sum_{k=1}^{n}\dfrac{(n-1)!}{(n-k)!}c(n-k,m-1),\\
	c(n,m)&=&c(n-1,m-1)+(n-1)c(n-1,m).\\
\end{eqnarray*}

Let us denote by $c_d(n,m)$ the permutations of $n$ elements with $m$ cycles all of which has length divisible by $d$. Clearly when $n$ is not divisible by $d$, $c_d(n,m)$ is zero. To find $c_d(n,m)$ when $n$ is divisible by $d$, consider the cycle containing $n$. If this cycle has length $dk$, then there are ${n−1 \choose dk−1}(dk-1)!=\dfrac{(n-1)!}{(n-dk)!}$ ways to choose this cycle, and $c_d(n−dk,m-1)$ ways to choose permutations consists of the $n−dk$ elements outside that cycle. Therefore, we have the following relation
$$	c_d(dn,m)= \sum_{k=1}^{t}\dfrac{(dn-1)!}{(dn-dk)!}c_d(dn-dk,m-1).$$
Another way to calculate the number $c_d(n,m)$ is to use the following recurrence relation
$$c_d(dn+d,m)=(dn+1)^{\overline{d-1}}c_d(dn,m-1)+(dn)^{\overline{d}}c_d(dn,m).$$ Indeed we can divide the permutations of $dn+d$ elements whose cycles are all divisible by $d$ into two types: the one in which the cycle containing $dn+d$ has length $d$ and the others. Note that there are $(dn+1)^{\overline{d-1}}c_d(dn,m-1)$ permutations of the first type since there are $(dn+d-1)(dn+d-2)\cdots (dn+1)=(dn+1)^{\overline{d-1}}$ ways to choose the other elements of the cycle containing $dn+d$. On the other hand, when the cycle containing $dn+d$ has length greater than $d$ by deleting the element $dn+d$ and the first $d-1$ elements coming next to it, we obtain a permutation of $dn$ elements that consists of exactly $m$ cycles of length divisible by $d$ and vice a versa. Here we can choose the elements that go next to $dn+d$ in the cycle containing $dn+d$ in $(dn+d-1)(dn+d-2)\cdots (dn+1)$ ways. This leaves $dn$ remaining elements. For any permutation of the remaining elements of the same type, we can choose one of these elements, say $x$ and place the $d$ elements in order to the left of $x$. Therefore the number of permutations of the second type is $(dn+d-1)(dn+d-2)\cdots (dn)c_d(dn,m)$ as desired.  
\begin{lemma}\label{lem:eqfordrising}For every $n\in \mathbb{N}$, we have the following relation \begin{eqnarray}\label{rising_even} (x)^{\bar{n}}=\dfrac{n!}{(dn)!}\overset{n}{\underset{m=0}{\sum}} c_d(dn,m)(xd)^m.\end{eqnarray}
\end{lemma}
\begin{proof} We prove by induction on $n$. The case $n=1$ is trivial. Since $c_d(dn,-1)=c_d(dn,n+1)=0$, we have 
\begin{eqnarray*}\dfrac{n!}{(dn+d)!}\overset{n+1}{\underset{m=0}{\sum}} c_d(dn+d,m)(xd)^m&=&\dfrac{n!}{d(dn)!}\overset{n+1}{\underset{m=0}{\sum}} c_d(dn,m-1)(xd)^m+\dfrac{n(n!)}{(dn)!}\overset{n+1}{\underset{m=0}{\sum}} c_d(dn,m)(xd)^m\\
	&=&\dfrac{x\cdot n!}{(dn)!}\overset{n}{\underset{m=0}{\sum}} c_d(dn,m)(xd)^m+\dfrac{n(n!)}{(dn)!}\overset{n}{\underset{m=0}{\sum}} c_d(dn,m)(xd)^m\\
	&=&(x+n)\dfrac{ n!}{(dn)!}\overset{n}{\underset{m=0}{\sum}} c_d(dn,m)(xd)^m
	\end{eqnarray*}
by the above formula. Since $(x+n)  x^{\overline{n}}=x^{\overline{n+1}}$, the result follows by induction.
\end{proof}

Let $c(n,m,e)$ denotes the number of permutation of $n$-elements with $m$-cycles exactly $e$ of them have even lengths.
Considering the cases where the cycle containing $n$ is even or odd, one can easily obtains the following formula
$$c(n,m,e)=\sum_{k=1}^{\lfloor \frac{n-1}{2}\rfloor}\dfrac{(n-1)!}{(n-2t-1)!}c(n-2t-1,m-1,e)+\sum_{k=1}^{\lfloor \frac{n}{2}\rfloor}\dfrac{(n-1)!}{(n-2t)!}c(n-2t,m,e-1).$$ Another recurrence relation including these numbers is
\begin{eqnarray}\label{giveneven_ones} c(n,m,e)=c(n-1,m-1,e)+(n-1)c(n-2,m-1,e-1)+(n-1)(n-2)c(n-2,m,e).\end{eqnarray} The above relation can be proved as above by considering the cases where the length of the cycle containing $n$ is $1$, $2$ or  $\geq 2$. Since $c(n,m,0)=c(n-1,m-1,0)+(n-1)(n-2)c(n-2,m,0)$, the following result easily follows from induction.
\begin{lemma}\label{eq:forallodd} For every $n\in \mathbb{N}$, $\underset{m=1}{\overset{n}{\sum}}2^mc(n,m,0)=2(n!).$
\end{lemma}

We also need the following results to count the number of weakly $\mathbb{Z}_2^n$-equivariant homeomorphism classes of small covers over a product of three simplices.

\begin{lemma}For every $n\in \mathbb{N}$, $$\sum_{m=1}^n2^m\sum_{e=1}^mc(n,m,e)=(n-1)(n!).$$
\end{lemma}
\begin{proof}The result follows from the relation $\overset{m}{\underset{e=1}{\sum}}c(n,m,e)=c(n,m)-c(n,m,0)$ and the corresponding relations for $c(n,m)$ and $c(n,m,0)$.
\end{proof}

\begin{proposition}\label{prop:mandev} For every $n\in \mathbb{N}$, we have the following formula
	\begin{eqnarray*}\sum_{m=1}^n2^m\sum_{e=1}^m2^ec(n,m,e)=
		\begin{cases}
			(2k)!(k^2+2k-1) \quad& \text{ when $n=2k$,} \\
			(2k+1)!k(k+3) \quad& \text{when $n=2k+1$ }.	\end{cases}
	\end{eqnarray*}
	
\end{proposition}
\begin{proof} We prove by induction. The case $n=1$ is trivial. Suppose that the above equations holds for all integers less than $n$. Using the equation (\ref{giveneven_ones}) and the Lemma \ref{eq:forallodd}, one can obtain that
\begin{eqnarray*}\sum_{m=1}^n2^m\sum_{e=1}^m2^ec(n,m,e)&=&2\sum_{m=1}^{n-1}2^m\sum_{e=1}^m2^ec(n-1,m,e)+\Big(4(n-1)+(n-1)(n-2)\Big)\Big(\sum_{m=1}^{n-2}2^m\sum_{e=1}^m2^ec(n-2,m,e)\Big)\\
	&+&4(n-1)\sum_{m=1}^{n-2}2^mc(n-2,m,0)\\
	&=&2\sum_{m=1}^{n-1}2^m\sum_{e=1}^m2^ec(n-1,m,e)+(n-1)(n+2)\Big(\sum_{m=1}^{n-2}2^m\sum_{e=1}^m2^ec(n-2,m,e)\Big)+8(n-1)!.
	\end{eqnarray*}
Now suppose that $n=2k$. By induction, we have 
\begin{eqnarray*}\sum_{m=1}^n2^m\sum_{e=1}^m2^ec(n,m,e)&=&2(2k-1)!(k-1)(k+2)+(2k-1)(2k+2)(2k-2)!((k-1)^2+2(k-1)-1)+8(2k-1)!\\
	&=&(2k)!(k^2+2k-1), \end{eqnarray*}
as desired. The case $n=2k+1$ can be shown similarly.
\end{proof}

As an immediate consequences of above results, we have the following
\begin{corollary}\label{cor:mandev-1} For every $n\in \mathbb{N}$, we have
	$$\sum_{m=1}^n2^m\sum_{e=1}^m(2^e-1)c(n,m,e)=(n!) \Big\lceil \frac{n}{2} \Big\rceil \Big\lfloor \dfrac{n}{2} \Big\rfloor.$$
\end{corollary}

\section{Products of three simplices}
\label{sect:Products of three simplices}

In this section, we give a formula for the number of weakly $\Z$-equivariant small covers over $P=\Delta^{n_1}\times \Delta^{n_2}\times\Delta^{n_3}$ where $n_1\leq n_2\leq n_3$ and $n_1+n_2+n_3=n$. By the Theorem \ref{thm:main}, the number of such classes is equivalent to the number of $\omega$-equivalence classes of acyclic $\omega$-vector weighted digraphs on $3$-labeled vertices $\{v_1,v_2,v_3\}$ where $\omega(i)=n_i, \ 1\leq i\leq 3$. Since the number of acyclic $\omega$-vector weighted digraphs depends also on the number of vertices whose images under the dimension function are the same, we need to consider the cases where $n_1<n_2<n_3$, $n_1=n_2<n_3$ and $n_1=n_2=n_3$, separately. 

We first consider the case where $n_1<n_2<n_3$. The others follow easily from this case. Note there are $25$ different acyclic digraph with labeled vertices $\{v_1,v_2,v_3\}$ and hence we can classify the acylic $\omega$-vector weighted digraphs as shown in the Figure \ref{figure:table3}. By duality, it suffices to understand the number of $\omega$-equivalence classes of Type $1$, Type $2$, Type $8$, Type $11$, Type $17$ and Type $23$. As a set, an equivalence class of the $\omega$-vector weighted digraphs of the type $1$, $2$, $8$ or $17$ consists of digraphs of the same type. However, $\omega$-vector weighted digraph of Type $11$ can be $\omega$-equivalent to that of Type $23$. There is only one $\omega$-vector weighted digraph of Type $1$. As discussed in the Example \ref{ex}, there are $\lfloor \frac{ n_1+1}{2}\rfloor$ different $\omega$-equivalence classes of Type $2$. Clearly, two $\omega$-vector weighted digraph of Type $17$ are $\omega$-equivalent if and only if the number of zero coordinates of the weight vectors $v$ and $w$ in each of the graphs are either the same or their sum is $n_2-1$ and $n_3-1$, respectively . Therefore the number of $\omega$-equivalence classes of Type $17$ is $\lfloor \frac{n_2+1}{2}\rfloor \cdot \lfloor \frac{n_3+1}{2}\rfloor$. 

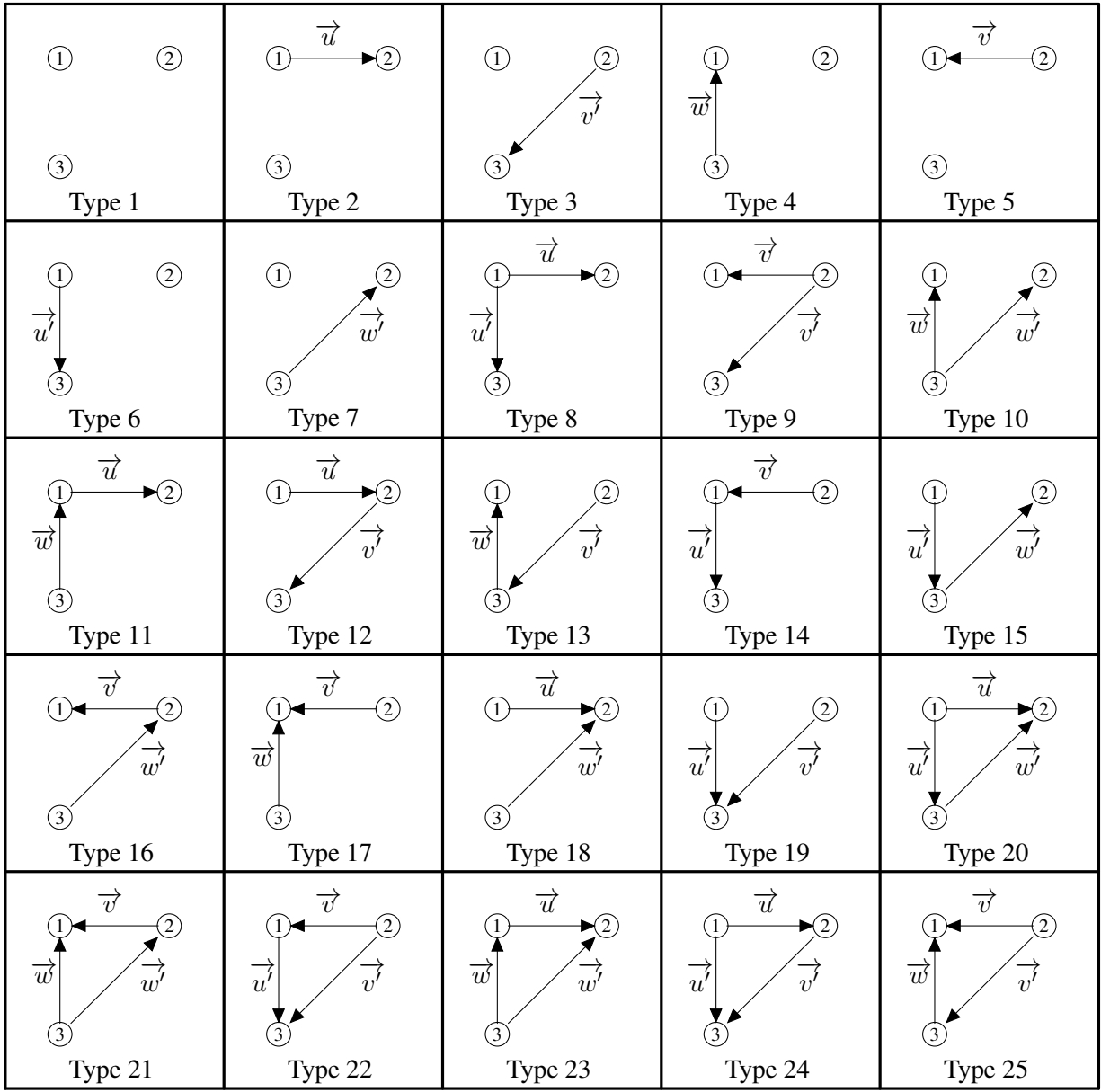
\begin{figure}[h!]
	\centering
\begin{tikzpicture}[line cap=round,line join=round,>=triangle 45,x=.8cm,y=.8cm]
\clip(-0.2,-0.2) rectangle (20.2,20.2);

\draw [->,line width=0.4pt] (1,1.2) -- (1,2.8);
\draw [->,line width=0.4pt] (1,14.8) -- (1,13.2);
\draw [->,line width=0.4pt] (1,9.2) -- (1,10.8);
\draw [->,line width=0.4pt] (1.2,11) -- (2.8,11);
\draw [->,line width=0.4pt] (1.2,5.2) -- (2.8,6.8);
\draw [->,line width=0.4pt] (1.2,1.2) -- (2.8,2.8);

\draw [->,line width=0.4pt] (2.8,3) -- (1.2,3);
\draw [->,line width=0.4pt] (2.8,7) -- (1.2,7);

\draw [->,line width=0.4pt] (5.2,19) -- (6.8,19);
\draw [->,line width=0.4pt] (5.2,11) -- (6.8,11);
\draw [->,line width=0.4pt] (5,5.2) -- (5,6.8);
\draw [->,line width=0.4pt] (5,2.8) -- (5,1.2);
\draw [->,line width=0.4pt] (5.2,13.2) -- (6.8,14.8);

\draw [->,line width=0.4pt] (6.8,3) -- (5.2,3);
\draw [->,line width=0.4pt] (6.8,2.8) -- (5.2,1.2);
\draw [->,line width=0.4pt] (6.8,10.8) -- (5.2,9.2);
\draw [->,line width=0.4pt] (6.8,7) -- (5.2,7);

\draw [->,line width=0.4pt] (9.2,15) -- (10.8,15);
\draw [->,line width=0.4pt] (9,14.8) -- (9,13.2);
\draw [->,line width=0.4pt] (9,9.2) -- (9,10.8);
\draw [->,line width=0.4pt] (9,1.2) -- (9,2.8);
\draw [->,line width=0.4pt] (9.2,3) -- (10.8,3);
\draw [->,line width=0.4pt] (9.2,1.2) -- (10.8,2.8);
\draw [->,line width=0.4pt] (9.2,7) -- (10.8,7);
\draw [->,line width=0.4pt] (9.2,5.2) -- (10.8,6.8);

\draw [->,line width=0.4pt] (10.8,18.8) -- (9.2,17.2);
\draw [->,line width=0.4pt] (10.8,10.8) -- (9.2,9.2);

\draw [->,line width=0.4pt] (13.2,3) -- (14.8,3);
\draw [->,line width=0.4pt] (13,17.2) -- (13,18.8);
\draw [->,line width=0.4pt] (13,10.8) -- (13,9.2);
\draw [->,line width=0.4pt] (13,6.8) -- (13,5.2);
\draw [->,line width=0.4pt] (13,2.8) -- (13,1.2);

\draw [->,line width=0.4pt] (14.8,14.8) -- (13.2,13.2);
\draw [->,line width=0.4pt] (14.8,15) -- (13.2,15);
\draw [->,line width=0.4pt] (14.8,11) -- (13.2,11);
\draw [->,line width=0.4pt] (14.8,6.8) -- (13.2,5.2);
\draw [->,line width=0.4pt] (14.8,2.8) -- (13.2,1.2);

\draw [->,line width=0.4pt] (17,13.2) -- (17,14.8);
\draw [->,line width=0.4pt] (17.2,13.2) -- (18.8,14.8);
\draw [->,line width=0.4pt] (17,10.8) -- (17,9.2);
\draw [->,line width=0.4pt] (17.2,9.2) -- (18.8,10.8);
\draw [->,line width=0.4pt] (17,6.8) -- (17,5.2);
\draw [->,line width=0.4pt] (17.2,5.2) -- (18.8,6.8);
\draw [->,line width=0.4pt] (17.2,7) -- (18.8,7);
\draw [->,line width=0.4pt] (17,1.2) -- (17,2.8);

\draw [->,line width=0.4pt] (18.8,19) -- (17.2,19);
\draw [->,line width=0.4pt] (18.8,2.8) -- (17.2,1.2);
\draw [->,line width=0.4pt] (18.8,3) -- (17.2,3);

\draw [line width=1.2pt] (0.,20.)-- (20.,20.);
\draw [line width=1.2pt] (20.,20.)-- (20.,0.);
\draw [line width=1.2pt] (20.,0.)-- (0.,0.);
\draw [line width=1.2pt] (0.,0.)-- (0.,20.);
\draw [line width=1.2pt] (0.,16.)-- (20.,16.);
\draw [line width=1.2pt] (0.,12.)-- (20.,12.);
\draw [line width=1.2pt] (0.,8.)-- (20.,8.);
\draw [line width=1.2pt] (0.,4.)-- (20.,4.);
\draw [line width=1.2pt] (4.,20.)-- (4.,0.);
\draw [line width=1.2pt] (8.,20.)-- (8.,0.);
\draw [line width=1.2pt] (12.,20.)-- (12.,0.);
\draw [line width=1.2pt] (16.,20.)-- (16.,0.);

\draw (0.3,2.5) node[anchor=north west] {$\overrightarrow{w}$};
\draw (2.3,2.5) node[anchor=north west] {$\overrightarrow{w'}$};
\draw (1.5,3.8) node[anchor=north west] {$\overrightarrow{v}$};
\draw (4.3,2.5) node[anchor=north west] {$\overrightarrow{u'}$};
\draw (5.5,3.8) node[anchor=north west] {$\overrightarrow{v}$};
\draw (6.3,2.5) node[anchor=north west] {$\overrightarrow{v'}$};
\draw (8.3,2.5) node[anchor=north west] {$\overrightarrow{w}$};
\draw (9.5,3.8) node[anchor=north west] {$\overrightarrow{u}$};
\draw (10.3,2.5) node[anchor=north west] {$\overrightarrow{w'}$};
\draw (12.3,2.5) node[anchor=north west] {$\overrightarrow{u'}$};
\draw (13.5,3.8) node[anchor=north west] {$\overrightarrow{u}$};
\draw (14.3,2.5) node[anchor=north west] {$\overrightarrow{v'}$};
\draw (16.3,2.5) node[anchor=north west] {$\overrightarrow{w}$};
\draw (17.5,3.8) node[anchor=north west] {$\overrightarrow{v}$};
\draw (18.3,2.5) node[anchor=north west] {$\overrightarrow{v'}$};

\draw (1.5,7.8) node[anchor=north west] {$\overrightarrow{v}$};
\draw (2.3,6.5) node[anchor=north west] {$\overrightarrow{w'}$};
\draw (4.3,6.5) node[anchor=north west] {$\overrightarrow{w}$};
\draw (5.5,7.8) node[anchor=north west] {$\overrightarrow{v}$};
\draw (9.5,7.8) node[anchor=north west] {$\overrightarrow{u}$};
\draw (10.3,6.5) node[anchor=north west] {$\overrightarrow{w'}$};
\draw (12.3,6.5) node[anchor=north west] {$\overrightarrow{u'}$};
\draw (14.3,6.5) node[anchor=north west] {$\overrightarrow{v'}$};
\draw (16.3,6.5) node[anchor=north west] {$\overrightarrow{u'}$};
\draw (17.5,7.8) node[anchor=north west] {$\overrightarrow{u}$};
\draw (18.3,6.5) node[anchor=north west] {$\overrightarrow{w'}$};

\draw (0.3,10.5) node[anchor=north west] {$\overrightarrow{w}$};
\draw (1.5,11.8) node[anchor=north west] {$\overrightarrow{u}$};
\draw (5.5,11.8) node[anchor=north west] {$\overrightarrow{u}$};
\draw (6.3,10.5) node[anchor=north west] {$\overrightarrow{v'}$};
\draw (8.3,10.5) node[anchor=north west] {$\overrightarrow{w}$};
\draw (10.3,10.5) node[anchor=north west] {$\overrightarrow{v'}$};
\draw (12.3,10.5) node[anchor=north west] {$\overrightarrow{u'}$};
\draw (13.5,11.8) node[anchor=north west] {$\overrightarrow{v}$};
\draw (16.3,10.5) node[anchor=north west] {$\overrightarrow{u'}$};
\draw (18.3,10.5) node[anchor=north west] {$\overrightarrow{w'}$};

\draw (0.3,14.5) node[anchor=north west] {$\overrightarrow{u'}$};
\draw (6.3,14.5) node[anchor=north west] {$\overrightarrow{w'}$};
\draw (8.3,14.5) node[anchor=north west] {$\overrightarrow{u'}$};
\draw (9.5,15.8) node[anchor=north west] {$\overrightarrow{u}$};
\draw (13.5,15.8) node[anchor=north west] {$\overrightarrow{v}$};
\draw (14.3,14.5) node[anchor=north west] {$\overrightarrow{v'}$};
\draw (16.3,14.5) node[anchor=north west] {$\overrightarrow{w}$};
\draw (18.3,14.5) node[anchor=north west] {$\overrightarrow{w'}$};

\draw (5.5,19.8) node[anchor=north west] {$\overrightarrow{u}$};
\draw (10.3,18.5) node[anchor=north west] {$\overrightarrow{v'}$};
\draw (12.3,18.5) node[anchor=north west] {$\overrightarrow{w}$};
\draw (17.5,19.8) node[anchor=north west] {$\overrightarrow{v}$};

\draw (1,16.7) node[anchor=north west] {Type 1};
\draw (5,16.7) node[anchor=north west] {Type 2};
\draw (9,16.7) node[anchor=north west] {Type 3};
\draw (13,16.7) node[anchor=north west] {Type 4};
\draw (17,16.7) node[anchor=north west] {Type 5};
\draw (1,12.7) node[anchor=north west] {Type 6};
\draw (5,12.7) node[anchor=north west] {Type 7};
\draw (9,12.7) node[anchor=north west] {Type 8};
\draw (13,12.7) node[anchor=north west] {Type 9};
\draw (17,12.7) node[anchor=north west] {Type 10};
\draw (1,8.7) node[anchor=north west] {Type 11};
\draw (5,8.7) node[anchor=north west] {Type 12};
\draw (9,8.7) node[anchor=north west] {Type 13};
\draw (13,8.7) node[anchor=north west] {Type 14};
\draw (17,8.7) node[anchor=north west] {Type 15};
\draw (1,4.7) node[anchor=north west] {Type 16};
\draw (5,4.7) node[anchor=north west] {Type 17};
\draw (9,4.7) node[anchor=north west] {Type 18};
\draw (13,4.7) node[anchor=north west] {Type 19};
\draw (17,4.7) node[anchor=north west] {Type 20};
\draw (1,0.7) node[anchor=north west] {Type 21};
\draw (5,0.7) node[anchor=north west] {Type 22};
\draw (9,0.7) node[anchor=north west] {Type 23};
\draw (13,0.7) node[anchor=north west] {Type 24};
\draw (17,0.7) node[anchor=north west] {Type 25};
\begin{scriptsize}
\draw [color=black] (1,1) circle (5pt);
\draw[color=black] (1,1) node {3};
\draw [color=black] (1,3) circle (5pt);
\draw[color=black] (1,3) node {1};
\draw [color=black] (3,3) circle (5pt);
\draw[color=black] (3,3) node {2};
\draw [color=black] (5,1) circle (5pt);
\draw[color=black] (5,1) node {3};
\draw [color=black] (5,3) circle (5pt);
\draw[color=black] (5,3) node {1};
\draw [color=black] (7,3) circle (5pt);
\draw[color=black] (7,3) node {2};
\draw [color=black] (9,1) circle (5pt);
\draw[color=black] (9,1) node {3};
\draw [color=black] (9,3) circle (5pt);
\draw[color=black] (9,3) node {1};
\draw [color=black] (11,3) circle (5pt);
\draw[color=black] (11,3) node {2};
\draw [color=black] (13,1) circle (5pt);
\draw[color=black] (13,1) node {3};
\draw [color=black] (13,3) circle (5pt);
\draw[color=black] (13,3) node {1};
\draw [color=black] (15,3) circle (5pt);
\draw[color=black] (15,3) node {2};
\draw [color=black] (17,1) circle (5pt);
\draw[color=black] (17,1) node {3};
\draw [color=black] (17,3) circle (5pt);
\draw[color=black] (17,3) node {1};
\draw [color=black] (19,3) circle (5pt);
\draw[color=black] (19,3) node {2};
\draw [color=black] (1,5) circle (5pt);
\draw[color=black] (1,5) node {3};
\draw [color=black] (1,7) circle (5pt);
\draw[color=black] (1,7) node {1};
\draw [color=black] (3,7) circle (5pt);
\draw[color=black] (3,7) node {2};
\draw [color=black] (5,5) circle (5pt);
\draw[color=black] (5,5) node {3};
\draw [color=black] (5,7) circle (5pt);
\draw[color=black] (5,7) node {1};
\draw [color=black] (7,7) circle (5pt);
\draw[color=black] (7,7) node {2};
\draw [color=black] (9,5) circle (5pt);
\draw[color=black] (9,5) node {3};
\draw [color=black] (9,7) circle (5pt);
\draw[color=black] (9,7) node {1};
\draw [color=black] (11,7) circle (5pt);
\draw[color=black] (11,7) node {2};
\draw [color=black] (13,5) circle (5pt);
\draw[color=black] (13,5) node {3};
\draw [color=black] (13,7) circle (5pt);
\draw[color=black] (13,7) node {1};
\draw [color=black] (15,7) circle (5pt);
\draw[color=black] (15,7) node {2};
\draw [color=black] (17,5) circle (5pt);
\draw[color=black] (17,5) node {3};
\draw [color=black] (17,7) circle (5pt);
\draw[color=black] (17,7) node {1};
\draw [color=black] (19,7) circle (5pt);
\draw[color=black] (19,7) node {2};
\draw [color=black] (1,9) circle (5pt);
\draw[color=black] (1,9) node {3};
\draw [color=black] (1,11) circle (5pt);
\draw[color=black] (1,11) node {1};
\draw [color=black] (3,11) circle (5pt);
\draw[color=black] (3,11) node {2};
\draw [color=black] (1,13) circle (5pt);
\draw[color=black] (1,13) node {3};
\draw [color=black] (1,15) circle (5pt);
\draw[color=black] (1,15) node {1};
\draw [color=black] (3,15) circle (5pt);
\draw[color=black] (3,15) node {2};
\draw [color=black] (1,17) circle (5pt);
\draw[color=black] (1,17) node {3};
\draw [color=black] (1,19) circle (5pt);
\draw[color=black] (1,19) node {1};
\draw [color=black] (3,19) circle (5pt);
\draw[color=black] (3,19) node {2};
\draw [color=black] (5,9) circle (5pt);
\draw[color=black] (5,9) node {3};
\draw [color=black] (5,11) circle (5pt);
\draw[color=black] (5,11) node {1};
\draw [color=black] (7,11) circle (5pt);
\draw[color=black] (7,11) node {2};
\draw [color=black] (5,13) circle (5pt);
\draw[color=black] (5,13) node {3};
\draw [color=black] (5,15) circle (5pt);
\draw[color=black] (5,15) node {1};
\draw [color=black] (7,15) circle (5pt);
\draw[color=black] (7,15) node {2};
\draw [color=black] (5,17) circle (5pt);
\draw[color=black] (5,17) node {3};
\draw [color=black] (5,19) circle (5pt);
\draw[color=black] (5,19) node {1};
\draw [color=black] (7,19) circle (5pt);
\draw[color=black] (7,19) node {2};
\draw [color=black] (9,19) circle (5pt);
\draw[color=black] (9,19) node {1};
\draw [color=black] (11,19) circle (5pt);
\draw[color=black] (11,19) node {2};
\draw [color=black] (9,17) circle (5pt);
\draw[color=black] (9,17) node {3};
\draw [color=black] (9,15) circle (5pt);
\draw[color=black] (9,15) node {1};
\draw [color=black] (11,15) circle (5pt);
\draw[color=black] (11,15) node {2};
\draw [color=black] (9,13) circle (5pt);
\draw[color=black] (9,13) node {3};
\draw [color=black] (9,11) circle (5pt);
\draw[color=black] (9,11) node {1};
\draw [color=black] (11,11) circle (5pt);
\draw[color=black] (11,11) node {2};
\draw [color=black] (9,9) circle (5pt);
\draw[color=black] (9,9) node {3};
\draw [color=black] (13,9) circle (5pt);
\draw[color=black] (13,9) node {3};
\draw [color=black] (13,11) circle (5pt);
\draw[color=black] (13,11) node {1};
\draw [color=black] (15,11) circle (5pt);
\draw[color=black] (15,11) node {2};
\draw [color=black] (13,13) circle (5pt);
\draw[color=black] (13,13) node {3};
\draw [color=black] (13,15) circle (5pt);
\draw[color=black] (13,15) node {1};
\draw [color=black] (15,15) circle (5pt);
\draw[color=black] (15,15) node {2};
\draw [color=black] (13,17) circle (5pt);
\draw[color=black] (13,17) node {3};
\draw [color=black] (13,19) circle (5pt);
\draw[color=black] (13,19) node {1};
\draw [color=black] (15,19) circle (5pt);
\draw[color=black] (15,19) node {2};
\draw [color=black] (17,9) circle (5pt);
\draw[color=black] (17,9) node {3};
\draw [color=black] (17,11) circle (5pt);
\draw[color=black] (17,11) node {1};
\draw [color=black] (19,11) circle (5pt);
\draw[color=black] (19,11) node {2};
\draw [color=black] (17,13) circle (5pt);
\draw[color=black] (17,13) node {3};
\draw [color=black] (17,15) circle (5pt);
\draw[color=black] (17,15) node {1};
\draw [color=black] (19,15) circle (5pt);
\draw[color=black] (19,15) node {2};
\draw [color=black] (17,17) circle (5pt);
\draw[color=black] (17,17) node {3};
\draw [color=black] (17,19) circle (5pt);
\draw[color=black] (17,19) node {1};
\draw [color=black] (19,19) circle (5pt);
\draw[color=black] (19,19) node {2};
\end{scriptsize}
\end{tikzpicture}
 \caption{Types of $\omega$-vector weighted digraphs with $3$ labeled vertices}
\label{figure:table3}
\end{figure}

\begin{lemma}
	The number of $\omega$-equivalence classes of Type $8$ is $\dfrac{2k^3+9k^2+k}{6}$  when $n_1=2k$ and $\dfrac{(k+1)(k^2+5k+3)}{3}$ when $n_1=2k+1$.
\end{lemma}

\begin{proof} We use Burnside's lemma. Let $X=\big(\mathbb{Z}_2^{n_1}\setminus \{\mathbf{0}\}\big)\times \big(\mathbb{Z}_2^{n_1}\setminus \{\mathbf{0}\}\big)$. Then the set of $\omega$-equivalence classes of Type $8$ is in one-to-one correspondence with the orbit space of the action of $S_{n_1+1}$ on $X$ defined by
\begin{eqnarray*}
	\sigma \cdot (v,w)
	=\begin{cases}
	 (\overline{\sigma}(v), \overline{\sigma}(w)) & \text{ if } v \in S_{\sigma} \ \text{and} \ w\in S_{\sigma}
	\\	
		(\overline{\sigma}(v)+e, \overline{\sigma}(w)) & \text{ if } v \notin S_{\sigma} \ \text{and} \ w\in S_{\sigma}
	\\
		(\overline{\sigma}(v), \overline{\sigma}(w)+e) & \text{ if } v \in S_{\sigma} \ \text{and} \ w\notin S_{\sigma}  \\
		(\overline{\sigma}(v)+e, \overline{\sigma}(w)+e) & \text{otherwise.} 
	\end{cases}
\end{eqnarray*} where $e=e_{\sigma^{-1}(n_1+1)}$ and $S_{\sigma}=\{v \in \mathbb{Z}_2^{n_1}\setminus \{\mathbf{0}\}|  (v)_{\sigma(n_1+1)}=0 \ \text{if} \ \sigma(n_1+1) \neq n_1+1\}$. Let $\lambda_1^{m_1}\cdots \lambda_k^{m_k}$ be the cycle type of $\sigma$ with $\lambda_1 <\lambda_2< \cdots< \lambda_k$ and $m=\underset{i}{\sum}m_i$. If $\sigma$ fixes $n_1+1$ then $(v,w)$ is fixed by $\sigma$ if and only if the values of the coordinates of $v$ and $w$ corresponding to the same cycle are the same. Therefore the number of fixed points of $\sigma$ is $(2^{m-1}-1)^2$ when $\sigma(n_1+1)=n_1+1$ since the cycle decomposition of $\overline{\sigma}=\lambda_1^{m_1-1}\cdots \lambda_k^{m_k}$ and at least one of the coordinates of $v$ and $w$ are non-zero.

Suppose that $\sigma(n_1+1)\neq n_1+1$. In this case the number of disjoint cycles of $\sigma$ and $\overline{\sigma}$ are the same.  Note that if $\overline{\sigma}(v)+e=v$ for some $v\in \mathbb{Z}_2^{n_1}\setminus \{\mathbf{0}\}$ then $\lambda_i$ must be even for $1\leq i\leq k$. Therefore if there is a cycle of odd length in the cycle decomposition of $\sigma$ and $(v,w)$ is fixed by $\sigma$ then $(v)_{\sigma(n_1+1)}=(w)_{\sigma(n_1+1)}=0$. In this case, we have $(v,w)=(\overline{\sigma}(v), \overline{\sigma}(w))$. Since all the coordinates of $v$ and $w$ corresponding to cycle containing $\sigma(n_1+1)$ are $0$, the number of fixed points of $\sigma$ is $(2^{m-1}-1)^2$ when $\lambda_i$ is odd for some $i$. On the other hand the number of elements $v\in S_{\sigma}$ satisfying the condition $\overline{\sigma}(v)+e=v$ is $2^{m-1}$. Therefore if all the $\lambda_i$ are even then $\sigma$ fixes $(2^{m-1}-1)^2$ elements in $S_{\sigma}\times S_{\sigma}$, $(2^{m-1})^2$ elements in $S_{\sigma}' \times S_{\sigma}'$ and $2(2^{m-1})(2^{m-1}-1)$ elements in $S_{\sigma}\times S_{\sigma}' \cup S_{\sigma}'\times S_{\sigma}$. Therefore there are $(2^m-1)^2$ elements of $X$ fixed by $\sigma$ when $\sigma$ contains cycles of even lengths only.

Therefore the number of $\omega$-equivalence classes of Type $8$ is given by the following formula
\begin{eqnarray*}&&\dfrac{1}{(n_1+1)!}\sum_{m=1}^{n_1+1} \Big((2^{m-1}-1)^2(c(n_1+1,m)-c_2(n_1+1,m))+ (2^m-1)^2c_2(n_1+1,m) \Big)\\ &=&\dfrac{1}{(n_1+1)!} \sum_{m=1}^{n_1+1}(2^{m-1}-1)^2c(n_1+1,m)+\dfrac{1}{(n_1+1)!}\sum_{m=1}^{n_1+1}(3\cdot 4^{m-1}-2^m)c_2(n_1+1,m)
	\end{eqnarray*}
by Burnside lemma. Since $(2^{m-1}-1)^2=\dfrac{4^m}{4}-2^m+1$, the first sum is equal to $
\dfrac{n_1^3+9n_1^2+2n_1}{24}$ by the formula (\ref{rising_Stirling}). Since $c_2(n_1+1,m)=0$ for even $n_1$, the number of $\omega$-equivalence classes of Type $8$ is $\dfrac{2k^3+9k^2+k}{6}$ when $n_1=2k$. When $n_1=2k+1$, the second sum is equal to $
\dfrac{3k+2}{4}$ by the Lemma \ref{lem:eqfordrising} and hence the number of $\omega$-equivalence classes of Type $8$ is $\dfrac{(k+1)(k^2+5k+3)}{3}$.
\end{proof}

Since $\omega$-vector weighted digraph of Type $11$ can only be $\omega$-equivalent to that of Type $11$ or Type $23$ and vice a versa, we need to consider their union that is obtained by allowing $w'$ to be zero in Type $23$. Using the same idea of the above proof, we obtain the following result.
\begin{lemma} The number of $w$-equivalence classes of Type $11$ and Type $23$ is equal to $h(n_1,n_3)$ where
\begin{eqnarray*} h(n,m)=\begin{cases} \displaystyle \frac{nm(m^2+9m+14)}{48} \quad & \text{ if $n$ and $m$ are even,} \\
		\displaystyle \frac{n(m^3+9m^2+23m+15)}{48} \quad & \text{ if $n, \ m$ are even and odd, respectively,} \\
		\displaystyle \frac{nm(m^2+9m+14)+3m(m+2)}{48} \quad & \text{ if $n, \ m$ are odd and even, respectively,} \\
		\displaystyle \frac{n(m^3+9m^2+23m+15)+3(m^2+2m-3)}{48} \quad & \text{ if $n$ is odd and $m\equiv 1$ (mod $4$),} \\
		\displaystyle \frac{n(m^3+9m^2+23m+15)+3(m^2+2m+1)}{48} \quad & \text{ if $n$ is odd and $m\equiv 3$ (mod $4$).}\end{cases}
		\end{eqnarray*}
\end{lemma}
\begin{proof} The number of $\omega$-equivalence classes of the union is equl to the size of the orbit space of $S_{n_1+1}\times S_{n_3+1}$-action on the set $X=\big(\mathbb{Z}_2^{n_1}\setminus \{\mathbf{0}\}\big) \times \big(\mathbb{Z}_2^{n_3}\setminus \{\mathbf{0}\}\big) \times \mathbb{Z}_2^{n_3}$ that is defined by
\begin{eqnarray*}
	(\sigma,\beta)\cdot (u,w,w')
	=\begin{cases}
	 (\overline{\sigma}(u), \overline{\beta}(w),\overline{\beta}(w') ) & \text{ if } u \in S_{\sigma}\ \text{and} \ w,w' \in S_{\beta}
	\\	
		(\overline{\sigma}(u), \overline{\beta}(w),\overline{\beta}(w')+e_2) & \text{ if } u \in S_{\sigma}, \ w \in S_{\beta}  \ \text{and} \ w'\notin S_{\beta}	\\
		(\overline{\sigma}(u), \overline{\beta}(w)+e_2,\overline{\beta}(w') ) & \text{ if } u \in S_{\sigma} \ w \notin S_{\beta}  \ \text{and} \ w'\in S_{\beta}\\
		(\overline{\sigma}(u), \overline{\beta}(w)+e_2,\overline{\beta}(w')+e_2) &\text{ if } u \in S_{\sigma} \ \text{and} \ w, w'\notin S_{\beta}\\
		(\overline{\sigma}(u)+e_1, \overline{\beta}(w),\overline{\beta}(w+w')) & \text{ if } u \notin S_{\sigma} \ \text{and} \ w,w'\in S_{\beta}
	\\	
		(\overline{\sigma}(u)+e_1, \overline{\beta}(w), \overline{\beta}(w+w')+e_2 ) & \text{ if } u \notin S_{\sigma} \ w \in S_{\beta}  \ \text{and} \ w'\notin S_{\beta}	\\
		(\overline{\sigma}(u)+e_1, \overline{\beta}(w)+e_2,\overline{\beta}(w+w')+e_2 ) & \text{ if } u \notin S_{\sigma} \ \ w \notin S_{\beta}  \ \text{and} \ w'\in S_{\beta}	\\
				(\overline{\sigma}(u)+e_1, \overline{\beta}(w)+e_2,\overline{\beta}(w+w')) & \text{otherwise.} 

	\end{cases}
\end{eqnarray*} where $e_1=e_{\sigma^{-1}(n_1+1)}$, $e_2=e_{\beta^{-1}(n_3+1)}$, $S_{\sigma}$ is defined as in the above proof, and $S_{\beta}=\{v \in \mathbb{Z}_2^{n_3}|  (v)_{\beta(n_3+1)}=0 \ \text{if} \ \beta(n_3+1) \neq n_3+1\}$.  We find the size of the orbit space of this action by calculating the number of fixed points of $(\alpha,\beta)$ as above. For this, we consider four cases depending on whether $\alpha$ and $\beta$ fix $n_1+1$ and $n_3+1$, respectively. Let the number of disjoint cycles of $\alpha$ and $\beta$ be $m_1$ and $m_2$, respectively and the number of cycles of even lengths in the cycle decompositions of $\alpha$ and $\beta$ be $e_1$ and $e_2$, respectively.\\

\noindent \textbf{Case 1:} Let $\alpha(n_1+1)=n_1+1$ and $\beta(n_3+1)=n_3+1$. We need to find pairs $(u,w,w')$ satisfying $\overline{\sigma}(u)=u, \overline{\beta}(w)=w$ and $\overline{\beta}(w')=w'$. As in the above proof, we can easily find that there are $(2^{m_1-1}-1)(2^{m_2-1}-1)2^{m_2-1}$ elements fixed by $(\sigma,\beta)$. Therefore the number of points fixed by a pair of this type is
\begin{eqnarray*} \sum_{m_1=1}^{n_1+1}\sum_{m_2=1}^{n_3+1}(2^{m_1-1}-1)(2^{m_2-1}-1)2^{m_2-1}c(n_1,m_1-1)c(n_3,m_2-1)=\dfrac{(n_1!)n_1(n_3+1)!n_3(n_3+5)}{6}.
	\end{eqnarray*}
by the formula (\ref{rising_Stirling}).\\

\noindent \textbf{Case 2:} Let $\alpha(n_1+1)=n_1+1$ and $\beta(n_3+1)\neq n_3+1$. The fixed points of a pair of this type can be counted as in the above Lemma by taking the cases where $w=0$ into an account. Hence the number of fixed points is equal to
\begin{eqnarray*}
	\begin{cases}
	(2^{m_1-1}-1)(4^{m_2}-2^{m_2}) \quad& \text{ when all the cycles of $\beta$ have even lengths,} \\
			(2^{m_1-1}-1)(4^{m_2-1}-2^{m_2-1}) \quad& \text{otherwise }.	\end{cases}
		\end{eqnarray*}
Note that the number of permutations that do not fix $n_3+1$ and have $m$ disjoint cycles is equal to $
n_3c(n_3,m)$. Since a permutation consists of cycles of even lengths only cannot fix an element, there are $c_2(n_3+1,m)$ permutations that do not fix $n_3+1$ and have $m$ disjoint cycles all of which have even lengths. Therefore the sum of number of fixed points of pairs of this type is
\begin{eqnarray*}&& \sum_{m_1=1}^{n_1+1}\sum_{m_2=1}^{n_3+1}(2^{m_1-1}-1)c(n_1,m_1-1)\Big((4^{m_2-1}-2^{m_2-1})(n_3c(n_3,m_2)-c_2(n_3+1,m_2))+(4^{m_2}-2^{m_2})c_2(n_3+1,m_2)\Big)\\
	&=& (n_1!)n_1 \Bigg(n_3\sum_{m_2=1}^{n_3+1}\big(4^{m_2-1}-2^{m_2-1})c(n_3,m_2)+\sum_{m_2=1}^{n_3+1}\big(3\cdot 4^{m_2-1}-2^{m_2-1}\big)c_2(n_3+1,m_2)\Bigg) \\
	&=&(n_1!)n_1 \Bigg(\frac{(n_3+1)!n_3(n_3+6)(n_3-1)}{24} + \frac{3}{4}\sum_{m_2=1}^{n_3+1} 4^{m_2}c_2(n_3+1,m_2)- \frac{1}{2}\sum_{m_2=1}^{n_3+1}2^{m_2}c_2(n_3+1,m_2) \Bigg)
\end{eqnarray*}

When $n_3$ is even, $c_2(n_3+1,m_2)=0$ and hence the above sum is equal to $\displaystyle (n_1!)n_1 \frac{(n_3+1)!n_3(n_3+6)(n_3-1)}{24}.$ Otherwise, we have 
$$\sum_{m_2=1}^{n_3+1} 4^{m_2}c_2(n_3+1,m_2)=(n_3+1)!\cdot(\frac{n_3+3}{2}) \ \mathrm{and} \ \ \sum_{m_2=1}^{n_3+1}2^{m_2}c_2(n_3+1,m_2)=(n_3+1)!$$ and hence the above sum is equal to $\displaystyle (n_1!)n_1 \frac{(n_3+1)!(n_3^3+5n_3^2+3n_3+15)}{24}.$ \\

\noindent \textbf{Case 3:} Let $\alpha(n_1+1)\neq n_1+1$ and $\beta(n_3+1)=n_3+1$. Clearly $(\alpha,\beta)$ fixes $(2^{m_1-1}-1)(2^{m_2-1}-1)2^{m_2-1}$ many $(u,w,w')$ for which $\alpha(n_1+1)$-th coordinate of $u$ is zero, that is $(u)_{\alpha(n_1+1)}=0$. Let us now consider $(u,w,w')$'s with $(u)_{\alpha(n_1+1)}\neq 0$, that is $u\notin S_{\alpha}$. In this case $\alpha(u)+e_1=u$ has a solution if and only if all the cycles of $\alpha$ have even lengths and there are $2^{m_1-1}-1$ many $u$'s satisfying this relation. We also need to find $(w,w')$ satisfying the equations $\overline{\beta}(w)=w$, $\overline{\beta}(w+w')$ and $w \neq 0$. Let $(i_1,i_2,\cdots,i_k)$ be a cycle of $\overline{\beta}$. By the first equation, $w_{i_1}=\cdots=w_{i_j}$, say $w_{i_1}=a$. By the second equation, we have $$w'_{i_j}=a+w_{i_{j}+1}, \ \text{for} \ 1\leq j\leq k-1 \ \text{and} \ w'_{i_k}=a+w_{i_1}.$$
Adding up these equations gives $ka \equiv 0 \ (\text{mod} \ 2)$. If $k$ is even, the matrix $(w|w')$ must consists of the blocks of one of the forms $\begin{pmatrix} 1 & \rvline & 0 \\ 1 & \rvline & 1 \end{pmatrix}, \begin{pmatrix} 1 & \rvline & 1 \\ 1 & \rvline & 0 \end{pmatrix}, \begin{pmatrix} 0 & \rvline & 0 \\ 0 & \rvline & 0 \end{pmatrix},$ or $\begin{pmatrix} 0 & \rvline & 1 \\ 0 & \rvline & 1 \end{pmatrix}$. If $k$ is odd then $a=0$ and hence the values of the coordinates of $w'$ corresponding to this cycle are either all $0$ or all $1$. Since $w \neq 0$, the number of $(w,w')$'s satisfying the above equations is $4^{e_2}2^{(m_2-e_2)-1}-2^{e_2}2^{(m_2-e_2)-1}=2^{m_2-1}(2^{e_2}-1).$ Therefore the number of the points fixed by $(\sigma,\beta)$ is
\begin{eqnarray*}
	\begin{cases}
		(2^{m_1-1}-1)(4^{m_2-1}-2^{m_2-1})+2^{m_1-1}2^{m_2-1}(2^{e_2}-1) \quad& \text{when all the cycles of $\alpha$ have even lengths,} \\
		(2^{m_1-1}-1)(4^{m_2-1}-2^{m_2-1}) \quad& \text{otherwise }.	\end{cases}
\end{eqnarray*}
Hence the number of all the elements of $X$ fixed by an element of this type is given by the following formula
\begin{eqnarray*}&=& n_1\sum_{m_1=1}^{n_1+1}\sum_{m_2=1}^{n_3+1}(2^{m_1-1}-1)(4^{m_2-1}-2^{m_2-1})c(n_1,m_1)c(n_3,m_2-1)\\&+&	\sum_{m_1=1}^{n_1+1}\sum_{m_2=1}^{n_3+1}\sum_{e_2=1}^{m_2}2^{m_1-1}2^{m_2-1}(2^{e_2}-1)c_2(n_1+1,m_1)c(n_3,m_2-1,e_2)\\
\end{eqnarray*}
The first term of this sum is $\displaystyle (n_1!)(n_3+1)!\frac{n_1(n_1-1)}{2}\frac{n_3(n_3+5)}{6}.$ The second term is zero when $n_1$ is an even number. When $n_1$ is an odd number, the second term depends on the parity of $n_3$ and we have
\begin{eqnarray*}
	\begin{cases}
		\displaystyle 	\frac{(n_1+1)!\cdot n_3!\cdot (n_3)^2}{8}\quad & \text{ if } n_3 \text{ is even number} \\
		\displaystyle \frac{(n_1+1)!\cdot (n_3+1)!\cdot (n_3-1)}{8} \quad & \text{ if } n_3 \text{ is odd number.} 
	\end{cases}
\end{eqnarray*} So the above sum is given by the following formula

\begin{eqnarray*}
	\begin{cases}
		\displaystyle   \frac {n_1! (n_1)(n_1-1)(n_3+1)!n_3(n_3+5)}{12} & \text{if $n_1$ is even.} \\
 \frac {n_1! (n_1)(n_1-1)(n_3+1)!n_3(n_3+5)}{12}+\frac{(n_1+1)!\cdot n_3!\cdot (n_3)^2}{8} & \text{if $n_1,n_3$ are odd and even, respectively.}\\
 \frac {n_1! (n_1)(n_1-1)(n_3+1)!n_3(n_3+5)}{12}+\frac{(n_1+1)!\cdot (n_3+1)!\cdot (n_3-1)}{8} & \text{if $n_1,n_3$ are odd and odd, respectively.}	
	\end{cases}
\end{eqnarray*}
\noindent \textbf{Case 4:} Let $\alpha(n_1+1)\neq n_1+1$ and $\beta(n_3+1)\neq n_3+1$. It follows similarly as in the proof of the above lemma that the number of $(u,w,w')$ satisfying $(u)_{\alpha(n_3+1)}=0$ and fixed by an element of this type is equal to
\begin{eqnarray*}
	\begin{cases}
		(2^{m_1-1}-1)(4^{m_2}-2^{m_2}) \quad& \text{ when all the cycles of $\beta$ have even lengths,} \\
		(2^{m_1-1}-1)(4^{m_2-1}-2^{m_2-1}) \quad& \text{otherwise }.	\end{cases}
\end{eqnarray*}
As above, $\overline{\sigma}(u)+e_1=u$ has a solution if and only if all the cycles of $\sigma$ has even lengths. From now on we assume that $\sigma$ consists of cycles of even lengths only and we solve the corresponding equations for $w$ and $w'$ depending on whether they are elements of $S_{\beta}$ or not. 

Let $(i_1.\cdots, i_k)$ be a cycle of $\beta$. If both $w$ and $w'$ are in $S_{\beta}$, we need to solve the equations
$$\overline{\beta}(w)=w, \ \overline{\beta}(w+w')=w', \ \text{and} \ (w)_{\beta(n_3+1)}=(w)_{\beta(n_3+1)}'=0.$$ By the first equation, we have $w_{i_1}=\cdots=w_{i_k}=a$ for some $a \in \{0,1\}$. If $\beta(n_3+1) \in \{i_1,\cdots,i_k\}$ then $w_{i_j}=w'_{i_j}=0$ for $1\leq j\leq k$ by the last equation. Otherwise $w'$ must satisfy the equations
$$w'_{i_j}=a+w'_{i_{j}+1}, \ \text{for} \ 1\leq j\leq k-1 \ \text{and} \ w'_{i_k}=a+w'_{i_1}.$$
As discussed in Case 3, there are $4$ solutions if $k$ is even and $2$ otherwise. Since $w$ can not be $0$, the number of $(u,w,w')$'s  with $u \notin S_{\sigma}$ and $w,w' \in S_{\beta}$ fixed by $(\alpha,\beta)$ is
\begin{eqnarray*}
	\begin{cases}
		2^{m_1-1}2^{m_2-1}(2^{e_2-1}-1) \quad& \text{when the the cycle containing $n_3+1$ has even length,} \\
		2^{m_1-1}2^{m_2-1}(2^{e_2}-1) \quad& \text{otherwise }.	\end{cases}
\end{eqnarray*}

If $w\in S_{\beta}$ and $w'\notin S_{\beta}$, we need to solve the equations
$$\overline{\beta}(w)=w, \ \ \overline{\beta}(w+w')+e_2=w', \ (w)_{\beta(n_3+1)}=0 \ \text{and} \ (w)_{\beta(n_3+1)}'=1.$$ Then $w_{i_1}=\cdots=w_{i_k}=a$ for some $a \in \{0,1\}$. If $\beta(n_3+1) \in \{i_1,\cdots,i_k\}$, say $\beta(n_3+1)=i_1$, then $a=0$ and $w'$ satisfies the equations 
$$w'_{i_j}=1+w'_{i_{j}+1}, \ \text{for} \ 1\leq j\leq k-1, \ \text{and} \ w'_{i_k}=w'_{i_1}=1.$$ Hence $k-1 \equiv 0 (\text{mod} \ 2)$, i.e, $k$ must be odd. When $k$ is odd, the above system has a unique solution.  Now suppose that $\beta(n_3+1) \notin \{i_1,\cdots,i_k\}$. By the second equation, we have $$w'_{i_j}=w_{i_{j}}+a+1, \ \text{for} \ 1\leq j\leq k-1 \ \text{and} \ w'_{i_k}=w_{i_1}+a+1.$$
Adding up these equations gives $k(a+1) \equiv 0 \ (\text{mod} \ 2)$. If $k$ is even, the matrix $(w|w')$ must consists of the blocks of one of the forms $\begin{pmatrix} 1 & \rvline & 0 \\ 1 & \rvline & 1 \end{pmatrix}, \begin{pmatrix} 1 & \rvline & 1 \\ 1 & \rvline & 0 \end{pmatrix}, \begin{pmatrix} 0 & \rvline & 0 \\ 0 & \rvline & 0 \end{pmatrix},$ or $\begin{pmatrix} 0 & \rvline & 1 \\ 0 & \rvline & 1 \end{pmatrix}$. If $k$ is odd then $a=1$ and hence the values of the coordinates of $w'$ corresponding to this cycle are either all $0$ or all $1$. Therefore the number of $(u,w,w')$'s  with $u \notin S_{\sigma}$, $w \in S_{\beta}$, and $w' \notin S_{\beta}$ fixed by $(\alpha,\beta)$ is
\begin{eqnarray*}
	\begin{cases}
        0 \quad& \text{when the the cycle containing $n_3+1$ has odd length,} \\
		2^{m_1-1}(4^{m_2-1}-2^{m_2-1}) \quad& \text{when all the cycles of $\beta$ have even lengths,} \\
		 2^{m_1-1}2^{m_2+e_2-2}\quad& \text{otherwise }.	\end{cases}
\end{eqnarray*}

Now suppose that $w\notin S_{\beta}$ and $w'\in S_{\beta}$. Then $(u,w,w')$ is a fixed points of $(\alpha,\beta)$ if
$$\overline{\beta}(w)+e_2=w, \ \ \overline{\beta}(w+w')+e_2=w', \ (w)_{\beta(n_3+1)}=1 \ \text{and} \ (w)_{\beta(n_3+1)}'=0.$$ If $\beta(n_3+1) \in \{i_1,\cdots,i_k\}$, say $\beta(n_3+1)=i_1$, then $w$ must  satisfy the equations $w_{i_j}=1+w_{i_{j}}$ for $1\leq j\leq k-1$ and $w_{i_k}=w_{i_1}=1.$  Hence $k-1 \equiv 0 (\text{mod} \ 2)$, i.e, $k$ must be odd. Let $k=2k'+1$. Then $w_{i_j}$ is equal to $1$ if $j$ is odd and $0,$ otherwise. Therefore $w'$ must satisfies the equations $w_{i_{2j+1}}'=w_{i_{2j+2}}+w'_{i_{2j+2}}+1$, $w_{i_{2j+2}}'=w_{i_{2j+3}}+w'_{i_{2j+3}}$ for $1\leq j\leq k'-1$ and $w_{i_1}'=0, w_{i_{2k'+1}}=1$. This forces $k'$ to be odd and hence the cycle containing $n_{3}+1$ must be divisible by $4$. In this case, we have a unique solution. If $\beta(n_3+1) \notin \{i_1,\cdots,i_k\}$, we need to simultaneously solve the equations  $$w_{i_j}=w_{i_{j}+1)}+1, \ w'_{i_j}=w_{i_{j}+1)}'+w_{i_{j}+1)}+1  \ \text{for} \ 1\leq j\leq k-1 \ \text{and} \ w_{i_k}=w_{i_1}+1, \ w'_{i_k}=w'_{i_1}+1.$$
Algebraically manupilating as above, one can show that this system has a solution if and only if $k$ is divisible by $4$. In this case the matrix $(w|w')$ must consists of only the blocks of one of the forms $\begin{pmatrix} 1 & \rvline & 0 \\ 0& \rvline & 1\\ 1 & \rvline &  1 \\ 0 & \rvline & 0 \\ \end{pmatrix}, \begin{pmatrix} 1 & \rvline & 1 \\ 0& \rvline & 0\\ 1 & \rvline &  0 \\ 0 & \rvline & 1 \\ \end{pmatrix}, \begin{pmatrix} 0 & \rvline & 0 \\ 1& \rvline & 0\\ 0 & \rvline &  1 \\ 1 & \rvline & 1 \\ \end{pmatrix}$ or $\begin{pmatrix} 0 & \rvline & 1 \\ 1& \rvline & 1\\ 0 & \rvline &  0 \\ 1 & \rvline & 0 \\ \end{pmatrix}$.  Therefore the number of $(u,w,w')$'s  with $u \notin S_{\sigma}$, $w \in S_{\beta}$, and $w' \notin S_{\beta}$ fixed by $(\alpha,\beta)$ is
\begin{eqnarray*}
	\begin{cases}
		2^{m_1-1}4^{m_2-1} \quad& \text{when all the cycles of $\beta$ have lengths divisible by $4$,} \\
		0 \quad& \text{otherwise }.	\end{cases}
\end{eqnarray*}

The last case we need to consider is the one where neither $w$ nor $w'$ are in $S_{\beta}$. To be a fixed point, $w$ and $w'$ must satisfy the equations
$$\overline{\beta}(w)+e_2=w, \ \ \overline{\beta}(w+w')=w',  \ \text{and} \ (w)_{\beta(n_3+1)}=(w)_{\beta(n_3+1)}'=1.$$
If $\beta(n_3+1) \in \{i_1,\cdots,i_k\}$, say $\beta(n_3+1)=i_1$, then $w$ must  satisfy the equations $w_{i_j}=1+w_{i_{j}}$ for $1\leq j\leq k-1$ and $w_{i_k}=w_{i_1}=1.$ This system has a solution only if $k$ is odd, say $k=2k'+1$. Since the solution is $w_{i_j}=1$ if $j$ is odd, and $w_{i_{j+1}}=0$, otherwise, $w'$ must satisfies the equations $w_{i_j}'=w_{i_{j+1}}'+1$,if $j$ is odd, $w'(i_j)=w'_{i_{j+1}}$, otherwise. This system has a solution if and only if $k'$ is odd, i.e, the cycle containing $n_3+1$ is divisible by $4$. Otherwise, $w$ and $w'$ must simultanesously satisfies the equations $w_{i_j}=w_{i_j+1},$ and $w'_{i_j}=w'_{i_{j+1}}+w'_{i_{j+1}}$. As before such a system has a solution if and only if $k$ is divisible by $4$. In this case the matrix $(w|w')$ must consists of only the blocks of one of the forms $\begin{pmatrix} 1 & \rvline & 0 \\ 0& \rvline & 0\\ 1 & \rvline &  1 \\ 0 & \rvline & 1 \\ \end{pmatrix}, \begin{pmatrix} 1 & \rvline & 1 \\ 0& \rvline & 1\\ 1 & \rvline &  0 \\ 0 & \rvline & 0 \\ \end{pmatrix}, \begin{pmatrix} 0 & \rvline & 0 \\ 1& \rvline & 1\\ 0 & \rvline &  1 \\ 1 & \rvline & 0 \\ \end{pmatrix}$ or $\begin{pmatrix} 0 & \rvline & 1 \\ 1& \rvline & 0\\ 0 & \rvline &  0 \\ 1 & \rvline & 1 \\ \end{pmatrix}$.  Therefore the number of $(u,w,w')$'s  with $u \notin S_{\sigma}$, $w \in S_{\beta}$, and $w' \notin S_{\beta}$ fixed by an element of this type is
\begin{eqnarray*}
	\begin{cases}
		2^{m_1-1}4^{m_2-1} \quad& \text{when all the cycles of $\beta$ have lengths divisible by $4$,} \\
		0 \quad& \text{otherwise }.	\end{cases}
\end{eqnarray*}

To sum up, the number of elements fixed by $(\alpha,\beta)$ when $\alpha(n_1+1)\neq n_1+1$ and $\beta(n_3+1)\neq n_3+1$ is equal to
\begin{eqnarray*}
	\begin{cases}
		(2^{m_1-1}-1)(4^{m_2-1}-2^{m_2-1}) \quad& \text{when both $\beta$ and $\alpha$ contains a cycle of odd length,} \\
		(2^{m_1-1}-1)(4^{m_2}-2^{m_2}) \quad& \begin{split} \text{when} \ \alpha \ \text{contains a cycle of odd length and}\\ \text{all cycles of} \ \beta \ \text{have even lengths},\end{split} \\ 
		(2^{m_1-1}-1)(4^{m_2-1}-2^{m_2-1})+2^{m_1-1}2^{m_2-1}(2^{e_2}-1) \quad& \begin{split} \text{when} \ \beta \ \text{contains a cycle of odd length and}\\ \text{all cycles of} \ \alpha \ \text{have even lengths},\end{split} \\
		(2^{m_1-1}-1)(4^{m_2}-2^{m_2})+2^{m_1}(4^{m_2-1}-2^{m_2-1})+2^{m_1}4^{m_2-1} \quad& \begin{split} \text{all cycles of} \ \alpha \ \text{have even lengths and}\\ \text{the lengths of the cycles of} \ \beta \ \text{are divisible by} \ 4 ,\end{split} \\
		(2^{m_1-1}-1)(4^{m_2}-2^{m_2})+2^{m_1}(4^{m_2-1}-2^{m_2-1}) \quad& \text{otherwise }.	\end{cases}
\end{eqnarray*}
Therefore the number of elements of $X$ fixed by such $(\alpha,\beta)$'s is equal to $I=I_1+I_2+I_3+I_4+I_5+I_6$ where
\begin{small}
\begin{eqnarray*}I_1&=&\sum_{m_1=1}^{n_1+1}\sum_{m_2=1}^{n_3+1}(2^{m_1-1}-1)(4^{m_2-1}-2^{m_2-1})(n_1c(n_1,m_1)-c_2(n_1+1,m_1))(n_3c(n_3,m_2)-c_2(n_3+1,m_2))\\
	I_2&=&\sum_{m_1=1}^{n_1+1}\sum_{m_2=1}^{n_3+1}(2^{m_1-1}-1)(4^{m_2}-2^{m_2})(n_1c(n_1,m_1)-c_2(n_1+1,m_1))c_2(n_3+1,m_2)\\
	I_3&=&\sum_{m_1=1}^{n_1+1}\sum_{m_2=1}^{n_3+1}(2^{m_1-1}-1)(4^{m_2-1}-2^{m_2-1})c_2(n_1+1,m_1)(n_3c(n_3,m_2)-c_2(n_3+1,m_2))\\
	I_4&=&\sum_{m_1=1}^{n_1+1}\sum_{m_2=1}^{n_3+1}2^{m_1-1}2^{m_2-1}c_2(n_1+1,m_1)\Big(\Big(\sum_{e_2=1}^{m_2}(2^{e_2}-1)(c(n_3+1,m_2,e_2)-c(n_3,m_2-1,e_2)\Big)-(2^{m_2}-1)c_2(n_3+1,m_2)\Big)\\
	I_5&=&\sum_{m_1=1}^{n_1+1}\sum_{m_2=1}^{n_3+1}\Big((2^{m_1-1}-1)(4^{m_2}-2^{m_2})+2^{m_1}(4^{m_2-1}-2^{m_2-1})\Big)c_2(n_1+1,m_1)c_2(n_3+1,m_2)\\
	\text{and}\\
	I_6&=&\sum_{m_1=1}^{n_1+1}\sum_{m_2=1}^{n_3+1}2^{m_1}4^{m_2-1}c_2(n_1+1,m_1)c_4(n_3+1,m_2).\\
	\end{eqnarray*}
\end{small}
Let $f:\mathbb{Z}_+^3\rightarrow \mathbb{R}$ be the function defined by $\displaystyle f(x,n,d)=\frac{(\frac{1}{x})^{\overline{\frac{n+1}{d}}}}{(\frac{n+1}{d})!}.$ Then the above sums are given by the following formulas.

\begin{eqnarray*}
	I_1&=&\begin{cases}
\displaystyle		 \frac{n_1!(n_3+1)!(n_1^2-n_1)(n_3^3+5n_3^2-6n_3)}{48}  & \text{ if $n_1$ and $n_3$ are even,} \\
\displaystyle		 \frac{n_1!(n_3+1)!(n_1^2-n_1)(n_3^3+5n_3^2-9n_3+3)}{48} & \text{ if $n_1,n_3$ are even and odd, respectively,} \\
\displaystyle		 \frac{n_1!(n_3+1)! \big(n_1^2-2n_1-1+2(n_1+1) f(2,n_1,2)\big)(n_3^3+5n_3^2-6n_3)}{48} & \text{ if $n_1,n_3$ are odd and even, respectively,}\\
\displaystyle		 \frac{n_1!(n_3+1)! \big(n_1^2-2n_1-1+2(n_1+1) f(2,n_1,2)\big)(n_3^3+5n_3^2-9n_3+3)}{48} & \text{ if $n_1,n_3$ are odd.}
	\end{cases}\\
		I_2 &=& \begin{cases}
		0 \quad & \text{ if $n_3$ is even.} \\
\displaystyle	 \frac{n_1!(n_3+1)!(n_1^2-n_1)(n_3+1)}{4} \quad & \text{ if $n_1,n_3$ are even and odd, respectively,} \\
\displaystyle	 \frac{n_1!(n_3+1)!\big(n_1^2-2n_1-1+2(n_1+1) f(2,n_1,2)\big)(n_3+1)}{4} \quad & \text{ if $n_1,n_3$ are odd.}
	\end{cases}\\
	I_3&=&\begin{cases}
		0 \quad & \text{ if $n_1$ is even.} \\
\displaystyle		\frac{(n_1+1)!(n_3+1)!(1-2f(2,n_1,2))(n_3^3+5n_3^2-6n_3)}{48} \quad & \text{ if $n_1,n_3$ are odd and even respectively,} \\
\displaystyle		\frac{(n_1+1)!(n_3+1)!(1-2f(2,n_1,2))(n_3^3+5n_3^2-9n_3+3)}{48} \quad & \text{ if $n_1,n_3$ are odd.}
	\end{cases}\\
	I_4&=&\begin{cases}
		0\quad & \text{ if $n_1$ is even.} \\
\displaystyle		\frac{(n_1+1)!n_3!(n_3^3+n_3^2+2n_3)}{16} \quad & \text{ if $n_1,n_3$ are odd and even, respectively,} \\
\displaystyle		\frac{(n_1+1)!(n_3+1)!(n_3^2-2n_3+1)}{16} \quad & \text{ if $n_1,n_3$ are odd.}
	\end{cases}\\
	I_5&=&\begin{cases}
\displaystyle	(n_1+1)!(n_3+1)!\Bigg(\frac{3n_3+1-4(n_3+1)f(2,n_1,2)}{8} \Bigg) \quad & \text{ if $n_1,n_3$ are odd, } \\
	0 \quad & \text{ otherwise.}
	\end{cases}\\
	I_6&=&\begin{cases}
\displaystyle	\frac{(n_1+1)!(n_3+1)!}{4} \quad & \text{ if $n_1$ is odd and $n_3=4k+3$ for $k\in \mathbb{Z}_{\geq 0}.$}	\\
	0 \quad & \text{ otherwise.}
	\end{cases}
\end{eqnarray*}
Therefore their sum is equal to
\begin{eqnarray*}
	\sum_{i=1}^6I_i=\begin{cases}
		\displaystyle \frac{n_1!(n_3+1)!(n_1^2-n_1)(n_3^3+5n_3^2-6n_3)}{48} &\text{ if $n_1,n_3$ are even,} \\
	\displaystyle 	\frac{n_1!(n_3+1)!(n_1^2-n_1)(n_3^3+5n_3^2+3n_3+15)}{48} &\text{ if $n_1,n_3$ are even and odd, respectively,} \\
	\frac{(n_1+1)!n_3!(n_3^3+n_3^2+2n_3)}{16}+\frac{n_1!(n_3+1)!(n_1^2-n_1) (n_3^3+5n_3^2-6n_3)}{48} & \text{ if $n_1,n_3$ are odd and even, respectively,} \\
	\frac{(n_1+1)!(n_3+1)!(n_3^3+8n_3^2+3n_3+12)}{48}+\frac{n_1!(n_3+1)!(n_1^2-2n_1-1) (n_3^3+5n_3^2+3n_3+15)}{48} & \text{ if $n_1$ is odd and $n_3=4k+1$ for $k\in \mathbb{Z}_{\geq 0},$} \\
	\frac{(n_1+1)!(n_3+1)!(n_3^3+8n_3^2+3n_3+24)}{48}+\frac{n_1!(n_3+1)!(n_1^2-2n_1-1) (n_3^3+5n_3^2+3n_3+15)}{48} & \text{ if $n_1$ is odd and $n_3=4k+3$ for $k\in \mathbb{Z}_{\geq 0}.$}
	\end{cases}
\end{eqnarray*}

Hence the number of orbits of the action is given by the formula
\begin{eqnarray*}
	\begin{cases}
		\displaystyle \frac{n_1n_3(n_3^2+9n_3+14)}{48} \quad & \text{ if $n_1$ and $n_3$ are even,} \\
		\displaystyle \frac{n_1(n_3^3+9n_3^2+23n_3+15)}{48} \quad & \text{ if $n_1,n_3$ are even and odd, respectively,} \\
		\displaystyle \frac{n_1n_3(n_3^2+9n_3+14)+3n_3(n_3+2)}{48} \quad & \text{ if $n_1,n_3$ are odd and even, respectively,} \\
		\displaystyle \frac{n_1(n_3^3+9n_3^2+23n_3+15)+3(n_3^2+2n_3-3)}{48} \quad & \text{ if $n_1$ is odd and $n_3=4k+1$ for $k\in \mathbb{Z}_{\geq 0},$} \\
		\displaystyle \frac{n_1(n_3^3+9n_3^2+23n_3+15)+3(n_3^2+2n_3+1)}{48} \quad & \text{ if $n_1$ is odd and $n_3=4k+3$ for $k\in \mathbb{Z}_{\geq 0},$}
	\end{cases}
\end{eqnarray*}
by Burnside's Lemma.
\end{proof}

As an immediate result of the above calculations, we have the following theorem.

\begin{theorem}
	Let $P=\Delta^{n_1}\times\Delta^{n_2}\times \Delta^{n_3}$ with $n_1\leq n_2 \leq n_3.$ The number of weakly $\Z$-equivariant homeomorphism classes of small covers over $P$ is given by the following formulas
	\begin{eqnarray*}
		\begin{cases}
			1+\underset{i=1}{\overset{3}{\sum}}(2\lfloor \frac{n_i+1}{2}\rfloor+f(n_i))+\underset{1\leq i<j\leq 3}{\sum}\lfloor \frac{n_i+1}{2} \rfloor \lfloor \frac{n_j+1}{2}+\underset{1\leq i\neq j\leq 3}{\sum}h(n_i,n_j) \quad& \text{ if } n_1<n_2<n_3, \\
			1+\lfloor \frac{n+1}{2}\rfloor +\lfloor \frac{n_3+1}{2}\rfloor+f(n)+f(n_3)+\lfloor \frac{n+1}{2}\rfloor\lfloor \frac{n_3+1}{2}\rfloor+\lfloor \frac{n+1}{2}\rfloor^2+h(n,n)+h(n,n_3)+h(n_3,n)\quad& \text{ if }  n=n_1=n_2<n_3\\
			1+\lfloor \frac{n+1}{2}\rfloor+f(n)+\lfloor \frac{n+1}{2}\rfloor^2+h(n,n)\quad& \text{ if } n=n_1=n_2n_3\\
		\end{cases}
	\end{eqnarray*}
	where 
	\begin{eqnarray*}
		f(n)=\begin{cases}
			\dfrac{n^3+9n^2+2n}{24}\quad& \text{ if \ $n$ \ is even}, \\
			\dfrac{(n+1)(n^2+8n+3)}{24}\quad& \text{otherwise.} \\
		\end{cases}
	\end{eqnarray*}
	and $h(n,m)$ is as defined in the above lemma.
\end{theorem}
\section*{Acknowledgment}
This work is supported by The Scientific and Technological Research Council of Turkey (Grant No: TBAG/118F310).

\end{document}